\documentclass[a4paper]{amsart}
\usepackage[latin1]{inputenc}
\usepackage{amsfonts}
\usepackage{amsmath,latexsym,amssymb,amsfonts, amsthm}
\usepackage{esint}
\usepackage{cite}
\usepackage{graphicx}
\usepackage{amscd}
\usepackage{color}
\usepackage{bm}             
\usepackage{enumerate}
\usepackage[dvips]{epsfig}
\usepackage{psfrag}
\usepackage{epsfig}

\usepackage[
  hmarginratio={1:1},     
  vmarginratio={1:1},     
  textwidth=16cm,        
  textheight=21cm,
  heightrounded,          
]{geometry}

\usepackage{graphicx,color}
\usepackage[colorlinks]{hyperref}
\hypersetup{linkcolor=blue,citecolor=blue,filecolor=black,urlcolor=blue}

\newtheorem{theorem}{Theorem}

\newtheorem{proposition}[theorem]{Proposition}
\newtheorem{lemma}[theorem]{Lemma}
\theoremstyle{definition}
\newtheorem{remark}[theorem]{Remark}

\numberwithin{theorem}{section}

\numberwithin{equation}{section}

\newcommand{\beq}{\begin{equation}}
\newcommand{\eeq}{\end{equation}}

\newcommand{\R}{\mathbb{R}}
\newcommand{\N}{\mathbb{N}}

\newcommand{\mf}[1]{\mathbf{#1}}

\newcommand{\pa}{\partial}
\renewcommand{\S}{\mathbb{S}}

\renewcommand{\div}{\,{\rm div}\,}

\newcommand{\cC}{{\mathcal C}}   
\newcommand{\cD}{{\mathcal D}}   
\newcommand{\cE}{{\mathcal E}}

\newcommand{\dist}{{\rm dist}}
\newcommand{\supp}{{\rm supp}}

\newcommand{\weak}{\rightharpoonup}

\newcommand{\eps}{\varepsilon}

\DeclareMathOperator{\loc}{loc}

\renewcommand{\epsilon}{\varepsilon}

\author[N. Soave]{Nicola Soave}\thanks{}
\address{Nicola Soave \newline \indent
Dipartimento di Matematica,  Politecnico di Milano,  \newline \indent
Via Edoardo Bonardi 9, 20133 Milano, Italy}
\email{nicola.soave@gmail.com; nicola.soave@polimi.it}

\author[S. Terracini]{Susanna Terracini}\thanks{}
\address{Susanna Terracini \newline \indent
 Dipartimento di Matematica ``Giuseppe Peano'', Universit\`a di Torino, \newline \indent
Via Carlo Alberto, 10,
10123 Torino, Italy}
\email{susanna.terracini@unito.it}

\title[An anisotropic monotoncity formula, with applications]{An anisotropic monotoncity formula, with applications to some segregation problems}
\thanks{N. S. is partially supported by the INDAM - GNAMPA group. \\
Declarations of interest: none.}

\begin{document}

\begin{abstract}
We prove an Alt-Caffarelli-Friedman montonicity formula for pairs of functions solving elliptic equations driven by different ellipticity matrices in their positivity sets. As application, we derive Liouville-type theorems for subsolutions of some elliptic systems, and we analyze segregation phenomena for systems of equations where the diffusion of each density is described by a different operator. 
\end{abstract} 

\maketitle

\section{Introduction}

The Alt-Caffarelli-Friedman (ACF) monotonicity formula is a cornerstone in the theory of free-boundary problems with two or more phases. In its original formulation \cite{ACF}, it establishes that if $u,v \in H^1_{\loc}(B_R) \cap C(B_R)$ are non-negative, continuous, subharmonic functions with disjoint positivity sets, i.e.
\[
u,v \ge 0, \qquad  - \Delta u \le 0, \qquad -\Delta v \le 0, \qquad u \cdot v \equiv 0 \qquad \text{in $B_R \subset \R^N$},
\] 
then the functional
\begin{equation}\label{J ACF orig}
r \mapsto J(u,v, x_0,r)  =  \frac{1}{r^{4}} \int_{B_r(x_0)} \frac{|\nabla u|^2}{|x-x_0|^{N-2}}\,dx \int_{B_r(x_0)} \frac{|\nabla v|^2}{|x-x_0|^{N-2}}\,dx 
\eeq
is monotone non-decreasing for $0<r<\dist(x_0,\pa B_R)$. Here and in the rest of the paper $B_r(x_0)$ (resp. $S_r(x_0)= \pa B_r(x_0)$) denotes the Euclidean ball (resp. sphere) of center $x_0$ and radius $r>0$, and we simply write $B_r$ and $S_r$ if $x_0 =0$. 

The monotonicity formula was introduced in \cite{ACF}, as the key tool to prove the optimal Lipschitz regularity of solutions to a two-phase problem, and since then 
has been successfully applied in a number of different contexts. Several generalizations of the ACF formula are now available, tailored to deal with elliptic or parabolic equations with variable coefficients \cite{Caff88, CKajm}, and also equations with right hand side  \cite{CJK, EP, MP, Vel}; in this latter case, one obtains the so-called almost monotoncity formula. Moreover, a counterpart of the ACF formula is available also for the fractional Laplacian \cite{TTV,TVZ14, TVZ16} and for the $p$-Laplacian \cite{DK}. A common feature of all these contributions is that different phases satisfy equations driven by the same operator, as in the original ACF result. 

In this paper we address the case when, on the contrary, $u$ and $v$ satisfy equations involving different uniformly elliptic operators.  Only recently some related free boundary problems have been investigated in the literature. In \cite{AM12},  Andersson and Mikayelyan  prove partial regularity of the zero set of weak solutions to a quasilinear divergence problem at the jump. Kim, Lee and Shahgholian (\cite{KLS,KLS19}), are concerned with the regularity of the solutions and of the nodal set to equations with jump of conductivity. 
Moreover,  in the paper \cite{CDS18}, Caffarelli, De Silva and Savin deal with a two phase anisotropic problem in dimension $2$, and prove the Lipschitz regularity of the solutions;  finally we quote \cite{CPQT19}, where it is investigated the regularity of interfaces of a Pucci type segregation problem.

In \cite{KLS}, the authors  focus on the problem 
\[
-\div(a_+(x) \nabla u) \le 0, \qquad   -\div(a_-(x) \nabla v) \le 0 \qquad \text{in $B_R$}
\]
for different \emph{scalar} positive functions $a_\pm$. The fact that $a_\pm$ are different scalar functions makes the problem asymmetric, but essentially isotropic, and indeed the authors obtained a perturbed monotonicity formula for the same functional $J$ defined in \eqref{J ACF orig}. In contrast, we deal with a truly anisotropic two phases problem, thus assuming that $\div(A_1 \nabla u) \ge 0$ and $\div(A_2 \nabla v) \ge 0$ for two positive definite symmetric $N \times N$ matrices $A_1, A_2$ with constant coefficients. This makes our setting somehow similar to that of  \cite{AM12},  where the authors consider weak solutions to $\div(B(w) \nabla w) = 0$, where  $B(w) = (A-\textrm{Id}) \chi_{\{w>0\}} + \textrm{Id}$.  As far as we know, the following is the first monotonicity formula of ACF type specifically tailored for the anisotropic case.
After some transformations, we can always assume that $A_1=A$ is diagonal, with lowest eigenvalue equal to $1$, and $A_2$ is the identity (see the proof of Theorem \ref{thm: liou seg} below for more details), and we obtain the following result.

\begin{theorem}\label{thm: ACF}
Let $N \ge 2$, let $A \neq \textrm{Id}$ be a $N \times N$ diagonal matrix with diagonal entries 
\[
1=a_1 \le a_2 \le \dots \le a_N, 
\]
and let
\beq\label{def fun}
\Gamma_A(x):= \left( \sum_{i=1}^N \frac{x_i^2}{a_i} \right)^\frac{2-N}2. 
\eeq
Let $u,v \in H^1_{\loc}(B_R)$ be such that
\[
u,v \ge 0, \qquad  - \div(A \nabla u) \le 0, \qquad -\Delta v \le 0, \qquad u \cdot v \equiv 0 \qquad \text{in $B_R \subset \R^N$}.
\] 
There exists an exponent $\nu_{A,N} \in (0,2)$ depending on $A$ and on $N$ such that the functional
\[
\begin{split}
r \mapsto J(u,v, x_0,r) & =  \frac{1}{r^{2\nu_{A,N}}} \int_{B_r(x_0)} \langle A \nabla u, \nabla u \rangle \Gamma_A(x-x_0)\,dx \int_{B_r(x_0)} \frac{|\nabla v|^2}{|x-x_0|^{N-2}}\,dx 
\end{split}
\]
is monotone non-decreasing for $0<r<\dist(x_0,\pa B_R)$, $x_0 \in B_R$.
\end{theorem}

The exponent $\nu_{A,N}$ is explicitly given as the solution of an optimal partition problem, involving eigenvalues of Dirichlet forms on the unit sphere $\S^{N-1}$, as in the original ACF formula. While in the isotropic case $A=\textrm{Id}$ the optimal value is known to be equal to $2$, in the anisotropic case $A \neq \textrm{Id}$ we shall show that such a spectral optimal value $\nu_{A,N}$ is always smaller than $2$ (Lemma \ref{lem: on ov}). One may still wonder whether or not it is possible to replace $\nu_{A,N}$ with $2$, in the monotonicity formula, with a strategy different to ours, thus improving Theorem \ref{thm: ACF}. It is worthwhile noticing that the answer is negative in general: the optimal exponent in the anisotropic monotonicity formula is strictly smaller than $2$, at least for suitable choices of $A$. This marks a striking difference with the symmetric-isotropic case, and we refer to Remark \ref{rem: opt} for a detailed discussion on this point.

\medskip

One of the difficulties in the proof of Theorem \ref{thm: ACF} is that the natural domains of integration for integrals involving $\Gamma_A$ are the ellipsoids $\mathcal{E}_r(x_0):=\{|A^{-1/2} (x-x_0)| < r\}$ rather than Euclidean balls. However, using different domains of integration for the two factors of $J$,  prevents us from reducing the proof of the monotonicity formula to an optimal partition problem, since $\pa \mathcal{E}_r(x_0)$ and $\pa B_r(x_0)$ do not coincide. In order to overcome this obstruction, we introduce suitable weights in the various integrations by parts, in analogy with the approach used in  \cite{Kuk} to prove an Almgren monotonicity formula for variable coefficients operators by avoiding the use of radial deformations or Riemannian metric considerations.

Finally, we mention that the possibility of proving a monotonicity formula without assuming the continuity of the phases was already considered in the literature (for instance in \cite{Vel}).

\medskip

\noindent{\textbf{Applications to segregation problems.}} The asymptotic analysis of phase separation in reaction-diffusion systems with multiple phases is a relevant field of application of the ACF monotonicity formula, as highlighted in the recent literature,  starting from \cite{CTV02,CTV05}. In particular, the ACF monotonicity formula can be usefully applied in order to prove a priori bounds of the solutions, independent of the singular perturbation parameter. Typical examples of such singularly perturbed systems fit under the comprehensive model
\[
    - \Delta u_i = f_i(x,u_i) - \beta g_i (u_1, \dots, u_k) \qquad \text{in $\Omega \subset \R^N$,}
\]
where the elliptic operator $-\Delta$ and the functions $\beta g_i \ge 0$ describe, respectively, the diffusion process and the interaction between the densities, and can assume different shapes according to the underlying phenomena. The parameter $\beta>0$ describes the strength of the competition, and one is particularly interested in understanding the behavior of solutions in the singular limit $\beta \to +\infty$, which is the limit of strong competition leading to total segregation. The following particular cases have been widely investigated in light of their relevance both from the mathematical point of view, and from the physical/biological one: 
\begin{itemize}
\item[($i$)] the \emph{Lotka-Volterra quadratic interaction} $g_i(u_1, \dots, u_k) = u_i \sum_{j \neq i} b_{ij} u_j$, see \cite{CTV05, CKL, DWZtams, TT12, SZ15, TVZ19} and references therein. 
\item [($ii$)] the \emph{variational cubic interaction} $g_i(u_1, \dots, u_k) = u_i \sum_{j \neq i} b_{ij} u_j^2$, with $b_{ij}=b_{ji}$ (it possesses a gradient structure since $g_i = \pa_{u_i} G$, where $G(u_1, \dots, u_k) = \sum_{i, j \neq i} b_{ij} u_i^2 u_j^2$); see \cite{CL08, CTV02, CTVind, NTTV10, DWZjfa, TT12, SZ15, SZ16, STTZ16} and references therein.
\end{itemize}
Besides, we mention \cite{Q13, VZ14} and \cite{TVZ14, TVZ16, DT19} for analogue studies in fully nonlinear or nonlocal contexts; \cite{CPQ, STTZ18} for long-range interaction models; and \cite{WWnon} for partial results involving a wider class of interaction terms. 

Most of these results concern doubly-symmetric settings, in the sense that there is a symmetry both in the interaction terms ($b_{ij}=b_{ji}$) and in the diffusion processes governing the spread of the components (all the equations are driven by the same operator). Up to our knowledge, asymmetric problems have been studied only in \cite{CTV05, TVZ19} ( in the case of Lotka-Volterra interactions with $b_{ij} \neq b_{ji}$), and in \cite{WWnon} (very general, possibly asymmetric, interaction in dimension $N=2$). In particular, nothing was known if each density $u_i$ is driven by a different operator $L_i$, and in what follows we describe our main results in this framework. We shall treat separately both the Lotka-Volterra quadratic interactions, and the variational cubic ones.

\medskip

\noindent \textbf{Lotka-Volterra quadratic interactions.} Let $N,k \ge 2$ be positive integers, and let $\Omega \subset \R^N$ be a bounded smooth domain. We consider the system
\begin{equation}\label{lv completa}
\begin{cases}
L_i u_{i}= \beta u_{i} \sum_{j \neq i} b_{ij} u_{j}, \quad u_i>0  &\text{ in $\Omega$}\\
u_{i} = \varphi_i &\text{ on $\pa \Omega$},
\end{cases} \qquad i=1,\dots,k.
\end{equation}
The operators $L_i$ are of type $L_i= \div (A_i \nabla(\,\cdot\,))$, where $A_1,\dots,A_k$ are positive definite symmetric matrixes with constant coefficients. The coefficients $b_{ij}$ are positive, so that the system is competitive, and not necessarily symmetric. Regarding the boundary data $\varphi_i$, we suppose that they are the restriction on $\pa \Omega$ of $C^{1,\gamma}(\overline{\Omega})$ functions, for some $\gamma \in (0,1)$, with the property that $\varphi_i \cdot \varphi_j \equiv 0$ in $\overline{\Omega}$.


\begin{theorem}\label{thm: lv}
Let $\mf{u}_\beta =(u_{1,\beta}, \dots, u_{k,\beta})$ be a solution of \eqref{lv completa} at fixed $\beta>1$. There exists $\bar \nu \in (0,2)$ depending only on $A_1,\dots,A_k$ and on $N$ such that the following holds: for any $\alpha \in (0, \bar \nu/2)$, there exists $C>0$ independent of $\beta$ such that $\|\mf{u}_\beta\|_{C^{0,\alpha}(\overline{\Omega})} \le C$. 
Moreover, as $\beta \to +\infty$, we have that up to a subsequence
\[
\mf{u}_\beta \to \mf{u} \quad \text{in $C^{0,\alpha}(\overline{\Omega})$ and in $H^1_{\loc}(\Omega)$, for every $\alpha \in (0, \bar \nu/2)$},
\]
and the limit $\mf{u}$ is a vector of nonnegative functions satisfying 
\[
\begin{cases}
L_i u_i = 0 & \text{in $\{u_i>0\}$, $i=1,\dots,k$,} \\
u_i \cdot u_j \equiv 0 & \text{in $\Omega$, for every $i \neq j$,} \\
u_i = \varphi_i & \text{on $\pa \Omega$.}
\end{cases}
\]
\end{theorem}

\medskip

\noindent \textbf{Variational cubic interactions.} Let $N,k \ge 2$ be positive integers, and let $\Omega \subset \R^N$ be a bounded smooth domain. We consider the system 
\begin{equation}\label{be completa}
\begin{cases}
-L_i u_{i}=f_{i,\beta}(x,u_{i}) -\beta u_{i} \sum_{j \neq i} b_{ij} u_{j}^2, \quad u_i>0  &\text{ in $\Omega$}\\
u_{i,\beta} = 0 &\text{ on $\pa \Omega$},
\end{cases} \qquad i=1,\dots,k.
\end{equation}
As in the Lotka-Volterra case, we assume that $L_i= \div (A_i \nabla(\,\cdot\,))$, with $A_i$ positive definite, symmetric, with constant coefficients. Moreover, we assume that the functions $f_{i,\beta}: \Omega \times \R \to \R$ are continuous, and that the coupling coefficients are positive and symmetric: $b_{ij}= b_{ji}>0$, so that the system has a variational structure.


\begin{theorem}\label{thm: be}
Let $\mf{u}_\beta =(u_{1,\beta}, \dots, u_{k,\beta})$ be a solution of \eqref{lv completa} at fixed $\beta>1$. Suppose that $\{\mf{u}_\beta:\beta>1\}$ is uniformly bounded in $L^\infty(\Omega)$, and that $f_{i,\beta}$ maps bounded sets of $\Omega \times \R$ in bounded sets of $\R$, uniformly with respect to $\beta$. Then there exists $\bar \nu \in (0,2)$ depending only on $A_1,\dots,A_k$ and on $N$ such that the following holds: for every $\alpha \in (0, \bar \nu/2)$ there exists $C>0$ independent of $\beta$ such that $\|\mf{u}_\beta\|_{C^{0,\alpha}(\overline{\Omega})} \le C$. 
Moreover, up to a subsequence, we have that
\[
\mf{u}_\beta \to \mf{u} \quad \text{in $C^{0,\alpha}(\overline{\Omega})$ and in $H^1(\Omega)$, for every $\alpha \in (0, \bar \nu/2)$}.
\]
If $f_{i,\beta} \to f_i$ locally uniformly as $\beta \to +\infty$, then the limit function $\mf{u}$ satisfies
\[
\begin{cases}
-L_i u_i = f_i(x,u_i) & \text{in $\{u_i>0\}$, $i=1,\dots,k$,} \\
u_i \cdot u_j \equiv 0 & \text{in $\Omega$, for every $i \neq j$,} \\
u_i = 0 & \text{on $\pa \Omega$,}
\end{cases}
\]
and the domain variation formula
\beq\label{dom var lim 1}
\begin{split}
2\int_{\Omega}  \sum_i \left( \langle dY A_i \nabla u_{i}, \nabla u_{i} \rangle -  f_{i}(x,u_{i}) \langle \nabla u_{i},Y \rangle \right)  - \int_{\Omega} \div Y  \sum_i \langle A_i \nabla u_{i}, \nabla u_{i} \rangle = 0.
\end{split}
\eeq
\end{theorem}

Theorems \ref{thm: lv} and \ref{thm: be} can be considered as the perfect anisotropic counterpart of the main results in \cite{CTV05} and \cite{NTTV10}. The value $\bar \nu$ is given explicitly as the minimum of a finite number of optimal exponents appearing in Theorem \ref{thm: ACF} for different choices of $A$. In particular, given $A_1,\dots,A_k$, the value $\bar \nu$ is the same both for Theorems \ref{thm: lv} and \ref{thm: be}. The proofs of these results follow the blow-up strategy developed in \cite{CTV05, NTTV10}. In these contexts, the ACF monotonicity formula is crucially employed to obtain some Liouville-type theorems for the limit configuration in the blow-up.  

The results here are not stated in the broader setting, and some extensions could be proved by combining the method presented here with others already used in the literature. For instance, it would not be difficult to add a nonlinear term $f_{i,\beta}$ in system \eqref{lv completa}, or to obtain local interior estimates under no regularity or boundedness assumptions on $\Omega$. We refer the interested reader to \cite{STTZ16} for further generalizations. We preferred to treat the prototypical problems \eqref{lv completa} and \eqref{be completa}, in analogy with \cite{CTV05, NTTV10}, in order to emphasize the main differences and difficulties which one has to face when passing from the isotropic setting to the anisotropic one, without inessential technicalities.

\begin{remark}
In the setting of Theorem \ref{thm: lv}, the existence of $\mf{u}_\beta$ can be proved by using Leray-Schauder degree theory as in \cite[Theorem 2.1]{CTV05}, or fixed point arguments as in \cite[Theorem 4.1]{CPQ}. Regarding Theorem \ref{thm: be}, the existence of $\mf{u}_\beta$ can be proved by variational methods (minimization or min-max), under different assumptions of $f_{i,\beta}$.

It is by now well known that the assumption that $\{\mf{u}_\beta\}$ is uniformly bounded in $L^\infty(\Omega)$ in Theorem \ref{thm: be} is natural and very mild. For instance, it is satisfied by family of solutions sharing the same variational characterization, at each $\beta>1$ fixed. In Theorem \ref{thm: lv}, such an assumption is implicit, since it follows from the sign of $\mf{u}_\beta$, the subharmonicity, and the boundary conditions. 
\end{remark}

Once that Theorems \ref{thm: lv} and \ref{thm: be} are proved, it is natural to investigate the free-boundary problem arising in the limit: that is, to understand the regularity of the limit configuration $\mf{u}$ and of the associated nodal set  $\Gamma=\{u_i=0 \ \text{for every $i$}\}$. From this point of view, the local symmetric case is essentially understood as a consequence of the results in \cite{CTVind, CKL, CL08, NTTV10, TT12}: $\mf{u}$ is Lipschitz continuous, and $\Gamma$ is the union of $C^{1,\alpha}$-hypersurface of dimension $N-1$, up to a singular set of dimension $N-2$. Moreover, in a neighborhood of each point $x_0$ on the regular part of $\Gamma$ precisely two components of $\mf{u}$ are different from $0$, and their difference is smooth (reflection law). The anisotropic case offers a number of challenges, and will be the object of future investigations. Here we only address a simplified setting, and in particular a $2$ components Lotka-Volterra system, in order to understand the type or result we shall look at. We recall that for systems of two components it is always possible to suppose that $A_2=\textrm{Id}$, and that $A_1=A$ is a diagonal matrix with lowest eigenvalue equal to $1$. Thus, we define $\nu_{A,N}$ as in Theorem \ref{thm: ACF}. Moreover, if necessary replacing $u_1$ with $a_{21}/a_{12} u_1$, we can suppose to have symmetry of the coupling coefficients $a_{12}=a_{21}$. 

\begin{theorem}\label{thm: limit lv}
In the previous setting, let $\mf{u}=(u,v)$ be a limit profile for solutions to \eqref{lv completa}, given by Theorem \ref{thm: lv}. Then $w=u-v$ is a weak solution of the quasi-linear equation
\begin{equation}\label{qu eq}
\div(B(w) \nabla w) = 0 \quad \text{in $\Omega$},
\eeq
with $B(w) = (A-\textrm{Id}) \chi_{\{w>0\}} + \textrm{Id}$. We have that $w$ is $\alpha$-H\"older continuous for every exponent $\alpha \in (0, \nu_{A,N}/2)$. Moreover, $\mu = \Delta w^-$ is a positive and locally finite measure with support in $\{w=0\}$, and has $\sigma$-finite $(N-1)$-dimensional Hausdorff measure. Furthermore for $\mu$-a.e. $x \in \{w=0\}$ there exists $r>0$ such that $\{w=0\} \cap B_r(x)$ is a $C^{1,\alpha}$ graph.
\end{theorem} 

The theorem follows directly from the convergence in Theorem \ref{thm: lv} and the main result in \cite{AM12} concerning the nodal set of solutions of equations like \eqref{qu eq}. The regularity theory both for the solutions to \eqref{qu eq}, and for their nodal set, seems to be a difficult task. Up to our knowledge, it is only known that weak solutions are H\"older continuous for some exponent. Concerning the nodal set, the only available results are those in \cite{AM12}.

\begin{remark}
From equation \eqref{qu eq}, it is not difficult to deduce that limits of the Lotka-Volterra system \eqref{lv completa} with $2$ components satisfy the free-boundary condition 
\[
|\nabla u| \langle A \nu, \nu \rangle = |\nabla v|
\]
on the regular part of $\{u=0=v\}$; indeed, if $\omega \subset \subset \Omega$ and $\gamma=\{u=0=v\} \cap \omega$ is $C^1$, then 
\[
0 = \int_{\omega \cap \{u>0\}} \langle A \nabla u, \nabla \varphi \rangle - \int_{\omega \cap \{v>0\}}\langle \nabla v, \nabla \varphi\rangle = -\int_{\gamma} \varphi \left( \langle A \nabla u, \nu \rangle - \langle \nabla v, \nu_2 \rangle \right),
\] 
for every $\varphi \in C^\infty_c(\omega)$, with $\nu=-\nu_2 = \frac{\nabla v}{|\nabla v|} = - \frac{\nabla u}{|\nabla u|}$.

Instead, under mild additional assumptions on the nonlinear terms $f_{i,\beta}$, limits of gradient-type systems \eqref{be completa} with $2$ components satisfy the free-boundary condition 
\[
|\nabla u|^2 \langle A \nu, \nu \rangle = |\nabla v|^2
\]
on the regular part of $\{u=0=v\}$. This follows directly from the domain variation formula \eqref{dom var lim 1}, by reasoning as in \cite[Proposition 2.1]{DWZjfa}. Therefore, the limit classes for problems \eqref{lv completa} and \eqref{be completa} do not coincide. This is another interesting difference with respect to the analogue symmetric problems, where the limit profiles can be studied in a unified way, and in particular the free-boundary condition reads $|\nabla u| = |\nabla v|$ both for Lotka-Volterra and for variational interactions (we refer to \cite[Section 8]{TT12} for more details).
\end{remark}

\medskip

\noindent \textbf{Structure of the paper.} In Section \ref{sec: ACF} we prove the anisotropic monotonicity formula and some variants concerning non-segregated solutions of some competitive systems. In Section \ref{sec: liou} we deduce various Liouville-type theorems. Such theorems will be used in Section \ref{sec: lv} and \ref{sec: be}, which contain the proofs of Theorems \ref{thm: lv}, \ref{thm: limit lv} and \ref{thm: be}.  

\medskip

\noindent \textbf{Acknowledgements.} We are grateful to Daniela De Silva for useful discussions concerning the paper \cite{CDS18}.

\section{Anisotropic Alt-Caffarelli-Friedman monotonicity formula}\label{sec: ACF} 

Let $A$ be a positive definite $N \times N$ diagonal matrix with constant coefficients, with lowest eigenvalue equal to $1$:
\begin{equation}\label{def A}
A:=\textrm{diag}(a_1,\dots,a_N), \quad \text{with } 1=a_1 \le a_2 \le \dots \le a_N.
\end{equation}
We introduce at first the basic notation which will be used throughout this and the next sections.
\begin{itemize}
\item $\Gamma_A$ denotes the function defined in \eqref{def fun}. Notice that $\Gamma_A \equiv 1$ in dimension $N=2$, while for $N \ge 3$ it is the (multiple of) fundamental solution of $\div(A \nabla \cdot)$
(for the explicit expression of $\Gamma_A$, we refer to \cite[Chapter 5, pp. 214]{BeJoSc}). 
\item As in \cite{Kuk}, we define
\begin{equation}\label{def mu}
\mu(x) :=\langle A \frac{x}{|x|}, \frac{x}{|x|}  \rangle \quad \implies \quad  1 \le \mu \le a_N,
\end{equation}
where $\langle \cdot, \cdot \rangle$ denotes the Euclidean scalar product.
\item Let $\nu$ be the outer unit vector on a sphere $S_r(x_0)$. We consider the tangential gradient (computed with respect to the scalar product induced by $A$)
\beq\label{def tang}
\nabla_\theta^A \varphi := \nabla \varphi - \frac{\langle A \nabla \varphi, \nu \rangle}{\langle A\nu, \nu \rangle} \nu = \nabla \varphi - \frac{\langle A \nabla \varphi, \nu \rangle}{\mu(x-x_0)} \nu.
\eeq
In this way, the gradient can be splitted in its normal and tangential part as usual:
\beq\label{rel tan}
\langle A \nabla \varphi, \nabla \varphi \rangle= \langle A \nabla_\theta^A \varphi, \nabla_\theta^A \varphi \rangle + \frac{\langle A \nabla \varphi, \nu \rangle^2}{\mu(x-x_0)};
\eeq
notice that, in case $A=\textrm{Id}$, this identity boils down to $|\nabla \varphi|^2 = |\nabla_\theta \varphi|^2 + (\pa_\nu \varphi)^2$.
\item For $u \in H^1(\S^{N-1})$, we consider the optimal value
\[
\lambda(A,u) := \inf\left\{ \frac{\int_{\S^{N-1}}  \langle A \nabla_\theta^A \varphi, \nabla_\theta^A \varphi \rangle \,d\sigma   }{\int_{\S^{N-1}} \varphi^2 \mu \,d\sigma} 
\left| \begin{array}{l} \varphi \in H^1(\S^{N-1} \setminus \{0\}) \ \text{and}\\ \mathcal{H}^{N-1}(\{\varphi \neq 0\} \cap \{u=0\}) = 0.  \end{array}\right.
\right\}.
\]
where $d\sigma = d\sigma_x$ and $\mathcal{H}^{N-1}$ stay for the usual $(N-1)$-dimensional Hausdorff measure. Notice that, if $u$ is also continuous, then $\lambda(\textrm{Id},u)$ is the first eigenvalue of the Laplace-Beltrami operator with homogeneous Dirichlet boundary condition on the open set $\{\xi \in \S^{N-1}: u(\xi) >0\}$.
\item For $u \in H^1(S_r(x_0))$, we set $u_{x_0,r}(\xi) = u(x_0+r \xi) \in   H^1(S_1) \simeq H^1(\S^{N-1})$. 
\item We define $\gamma: \R^+ \to \R^+$
\[
\gamma(t) := \sqrt{\left(\frac{N-2}{2}\right)^2 + t\,} - \frac{N-2}{2}.
\]
\item For $N \ge 3$ and $\delta>0$, we define $\phi_\delta: [0,+\infty) \to (0,+\infty)$ and $\Phi_\delta: \R^N \to (0,+\infty)$ by
\begin{equation}\label{reg_fun}
\phi_\delta(r) =  \begin{cases}
\frac{N}2 \delta^{2-N} + \frac{2-N}2\delta^{-N} r^2 & \text{if }0 \le r \le \delta \\
r^{2-N} & \text{if $r>\delta$},
\end{cases} \qquad \Phi_\delta(x) = \phi_\delta(|x|).
\end{equation}
$\Phi_\delta$ is a $C^1$ positive superharmonic function in $\R^N$. Therefore, $\Phi_{A,\delta}(x) = \Phi_\delta(A^{-\frac12} x)$ is in turn a $C^1$ positive function in $\R^N$, with the properties that $\div(A \nabla \Phi_{A,\delta}) \le 0$, and $\Phi_{A,\delta}= \Gamma_A$ in the set $\mathcal{E}_\delta^c = \{|A^{-\frac12}x| \ge \delta\}$. The set $\cE_r:= \{|A^{-\frac12}x| < r\}$ is an ellipsoid, and, since $a_1=1$, 
\[
B_{r} \subset \cE_r \subset B_{r a_N^{1/2}}.
\]
\end{itemize}

\begin{remark}\label{rmk: tang grad}
It is convenient to observe that 
\[
\frac{a_2}{a_N} \int_{\omega} |\nabla_\theta \varphi|^2\,d\sigma\le \int_{\omega}  \langle A \nabla_\theta^A \varphi, \nabla_\theta^A \varphi \rangle \,d\sigma \le a_N a_{N-1}\int_{\omega} |\nabla_\theta \varphi|^2 \,d\sigma
\]
for any $\varphi \in H^1(\omega)$, for any $\omega \subset \S^{N-1}$, so that $\int_{\omega}  \langle A \nabla_\theta^A \varphi, \nabla_\theta^A \varphi \rangle$ is a semi-norm in $H^1(\omega)$, equivalent to the standard one $\int_{\omega} |\nabla_\theta \varphi|^2$. The above inequality can be easily checked as follows:
\begin{align*}
\int_{\omega}  \langle A \nabla_\theta^A \varphi, \nabla_\theta^A \varphi \rangle &= \int_{\omega} \langle A \nabla \varphi, \nabla \varphi \rangle - \frac{\langle A \nabla \varphi, \nu \rangle^2}{\langle A \nu, \nu \rangle}  \\
& = \int_{\omega} \sum_i a_i \varphi_{x_i}^2 - \frac{\left( \sum_i a_i \varphi_{x_i} x_i \right)^2}{\sum_i a_i x_i^2} \\
& = \int_{\omega}  \frac{1}{\sum_i a_i x_i^2}  \sum_{i < j} a_i a_j \left( \varphi_{x_i} x_j  -\varphi_{x_j} x_i\right)^2,
\end{align*}
and similarly
\[
\int_{\omega} |\nabla_\theta \varphi|^2 = \int_{\omega} \sum_{i<j}  \left( \varphi_{x_i} x_j  -\varphi_{x_j} x_i\right)^2.
\]
\end{remark}

\subsection{Monotonicity formula in dimension $N \ge 3$}

For $N \ge 3$, we define
\begin{equation}\label{def IA}
I_A (u,x_0,r) = \int_{B_r(x_0)} \langle A\nabla u, \nabla u \rangle \Gamma_A(x-x_0)\,dx.
\end{equation}

\begin{lemma}\label{lem: ACF pre}
Let $N \ge 3$, and let $u \in H^1_{\loc}(B_R)$ be nonnegative, and such that $\div(A \nabla u) \ge 0$ in $B_R$. Then, for almost every $r >0$ such that $B_r(x_0) \subset \subset B_R$, we have
\[
I_A(u,x_0,r) \le 
\frac{a_N^\frac{N}{2}   r}{2 \gamma( \lambda(A,u_{x_0,r})) } \int_{S_r(x_0)} \langle A\nabla u, \nabla u \rangle \Gamma_A(x-x_0) \, d\sigma.
\]
\end{lemma}

\begin{proof}
In order to simplify the notation, we consider $x_0=0$, and we often omit the dependence on $x_0$ and $A$ on most of the quantities. 

Let $u_\eps$ be a mollification of $u$, which still satisfies $\div(A\nabla u_\eps) \ge 0$ and $u_\eps \ge 0$. By using the coarea formula, it is not difficult to check that:
\begin{itemize}
\item[($i$)] for almost every $r \in (0,R)$ the restrictions of $u$ and of $\pa_{x_i} u$ ($i=1,\dots,k$) on $S_r$ are well defined, are in $L^2(S_r)$, and $u|_{S_r} \in H^1(S_r)$;
\item[($ii$)] for almost every $r \in (0,R)$ the restrictions of $u_\eps$ and of $\pa_{x_i} u_\eps$ on $S_r$ strongly converge to those of $u$ and of $\pa_{x_i} u$ in $L^2(S_r)$, as $\eps \to 0^+$.
\end{itemize}

We consider $r \in (0,R)$ such that both ($i$) and ($ii$) hold, and prove the lemma for these $r$. Let $\delta>0$ be such that $\{|A^{-1/2} x| \le \delta\} \subset \subset B_r$. In this way, we have that $\Phi_{A,\delta} = \Gamma_A$ in a neighborhood of $S_r$. We also recall that $\mu = \langle A \nu, \nu \rangle$ on $S_r$. By testing the equation for $u_\eps$ with $u_\eps \Phi_{A,\delta}$ in $B_r$, and recalling that $\div(A \nabla \Phi_{A,\delta}) \le 0$, we obtain
\begin{align*}
\int_{B_r} \langle A\nabla u_\eps, \nabla u_\eps \rangle \Phi_\delta & \le \int_{S_r} \Gamma u_\eps \langle A \nabla u_\eps, \nu \rangle  - \frac12 \int_{B_r}  \langle \nabla (u_\eps^2), A \nabla \Phi_\delta \rangle  \\
& \le \int_{S_r} \Gamma u_\eps \langle A \nabla u_\eps, \nu \rangle - \frac12 \int_{S_r} u_\eps^2 \langle A \nabla \Gamma,\nu\rangle. 
\end{align*}
By taking the limit as $\eps \to 0^+$ at first, and afterwards as $\delta \to 0^+$ (thanks to ($i$) and ($ii$)), we deduce that
\begin{equation}\label{l1-1}
I(u,r) \le \int_{S_r} u \Gamma \langle A \nabla u, \nu \rangle - \int_{S_r} \frac{u^2}2 \langle A \nabla \Gamma,\nu\rangle. 
\end{equation}
Now, on $S_r$ we have
\[
-\langle A \nabla \Gamma,\nu \rangle = (N-2) \left( \sum_i \frac{x_i^2}{a_i} \right)^{-\frac{N}2} |x| \le (N-2) \frac{a_N^{\frac{N}2}}{r^{N-1}}.
\]
Thus, recalling also that $\mu = |x|^{-2}\sum_i a_i x_i^2 \ge a_1 = 1$, we have
\beq\label{l1-2}
-\int_{S_r} \frac{u^2}2 \langle A \nabla \Gamma,\nu\rangle \le \frac{(N-2) a_N^{\frac{N}2}}{ 2 r^{N-1}} \int_{S_r} u^2 \mu. 
\eeq
Similarly, 
\beq\label{l1-3}
\begin{split}
\int_{S_r} \Gamma u \langle A \nabla u,\nu \rangle & \le \frac{a_N^\frac{N-2}{2}}{r^{N-2}} \int_{S_r} |u| |\langle A \nabla u,\nu \rangle|  \\ & \le \frac{a_N^\frac{N}{2}}{r^{N-2}} \left[\frac{\alpha}{2 r} \int_{S_r} u^2 \mu + \frac{r}{2\alpha} \int_{S_r} \frac{\langle A \nabla u,\nu \rangle^2}{\mu}\right],
\end{split}
\eeq
with $\alpha >0$ to be conveniently chosen later (in the last step, we used that $a_N \ge a_1=1$). By combining \eqref{l1-1}-\eqref{l1-3}, we obtain 
\[
\begin{split}
I(u,r) &\le \frac{a_N^\frac{N}{2}}{2 r^{N-3}} \left[ \frac{N-2+\alpha}{r^2} \int_{S_r} u^2 \mu + \frac{1}{\alpha} \int_{S_r} \frac{\langle A \nabla u,\nu \rangle^2}{\mu}\right] \\
& \le \frac{a_N^\frac{N}{2}}{2 r^{N-3}} \left[ \frac{N-2+\alpha}{\lambda(A,u_r)} \int_{S_r} \langle A \nabla_\theta^A u, \nabla_\theta^A u\rangle + \frac{1}{\alpha} \int_{S_r} \frac{\langle A \nabla u,\nu \rangle^2}{\mu}\right], 
\end{split}
\]
where we used the definition of $\lambda(A,u_r)$. We choose now $\alpha>0$ in order to perfectly balance the coefficients: that is, we impose
\[
\frac{N-2+\alpha}{\lambda(A,u_r)} = \frac{1}{\alpha} \quad \implies \quad \alpha = \gamma( \lambda(A,u_r) ).
\]
In this way we deduce that
\[
\begin{split}
I(u,r) & \le \frac{a_N^\frac{N}{2}}{2 \gamma( \lambda(A,u_r) ) r^{N-3}} \left[  \int_{S_r} \langle A \nabla_\theta^A u, \nabla_\theta^A u\rangle + \int_{S_r} \frac{\langle A \nabla u,\nu \rangle^2}{\mu}\right] \\
&= \frac{a_N^\frac{N}{2} r^{3-N}}{2 \gamma( \lambda(A,u_r))}  \int_{S_r} \langle A \nabla u, \nabla u \rangle \le \frac{a_N^\frac{N}{2} r}{2 \gamma( \lambda(A,u_r) )}  \int_{S_r} \langle A \nabla u, \nabla u \rangle \Gamma
\end{split}
\]
where we used \eqref{rel tan} and the fact that $\Gamma(x) \ge a_1^{(N-2)/2} |x|^{2-N} = r^{2-N}$ on $S_r$. 
\end{proof}

\begin{remark}\label{rmk acf pre}
Notice that the lemma is still valid for $A=\textrm{Id}$. Of course, in that case we have that $a_1=a_N=1$, and $\Gamma_A(x) = |x|^{2-N}$.
\end{remark}

Motivated by Lemma \ref{lem: ACF pre}, we study now the following asymmetric optimal partition problem:
\begin{equation}\label{opp}
\nu_{A,N}:=\inf\left\{ a_N^{-\frac{N}2}\gamma(\lambda(A,u)) + \gamma(\lambda(\textrm{Id}, v))\left| \begin{array}{l} u,v \in H^1(\S^{N-1})  \\  
\int_{\S^{N-1}} u^2 v^2 = 0.
\end{array}\right.\right\},
\end{equation}
with the convention that $\lambda(A,u) = +\infty$ if $u \equiv 0$ on $\S^{N-1}$ (this gives some continuity to $\lambda(A,\,\cdot)$ since, if $\mathcal{H}^{N-1}(\{u_n>0\}) \to 0$, then 
$\lambda(A,u_n) \to +\infty$ by the Sobolev inequality). The infimum is greater than or equal to $0$, since we are minimizing the sum of two non-negative quantities. We also recall that, in the symmetric case $A=\textrm{Id}$, it is known that $\nu_{\textrm{Id},N}=2$ (Friedland-Hayman inequality\footnote{The Friedland-Hayman inequality is usually syayed in a slightly different form, involving partitions of the sphere in disjoint open sets. However, in light of the condition $\int_{\S^{N-1}} u^2 v^2 = 0$ in the definition of $\nu_{A,N}$, it is not difficult to check that the inequality is equivalent to the fact that $\nu_{\textrm{Id},N}=2$.}), and the optimal value is reached if and only if $u$ and $v$ are $1$-homogeneous functions supported on disjoint half-spherical caps (see \cite[Chapter 2]{PSU12} and references therein for more details). Differently to the symmetric case, we are not able to characterize $\nu_{A,N}$ or to classify the optimizers. We are only able to exclude that $\nu_{A,N} =0$, and to bound it from above.

\begin{lemma}\label{lem: on ov}
Let $A \neq \textrm{Id}$ be a matrix as in \eqref{def A}. Then $0 <\nu_{A,N} <2$. 
\end{lemma}

\begin{proof}
We prove at first that $\nu_{A,N}>0$. Suppose by contradiction that $\nu_{A,N} = 0$, and let $(u_{n}, v_{n})$ be a minimizing sequence. By definition of $\gamma$, this implies that there exist $(u_n,v_n) \in H^1(\S^{N-1}) \times H^1(\S^{N-1})$ such that
\[
\begin{split}
\int_{\S^{N-1}} \langle A \nabla_\theta^A u_n, \nabla_\theta^A u_n \rangle \to 0, \qquad \int_{\S^{N-1}}|\nabla_\theta v_n|^2 \to 0, \\ \int_{\S^{N-1}} u_n^2 \mu \equiv 1, \qquad \int_{\S^{N-1}} v_n^2 \equiv 1,  \qquad \int_{\S^{N-1}} u_n^2 v_n^2 \equiv 0.
\end{split}
\]
Recalling Remark \ref{rmk: tang grad}, we deduce that up to a subsequence $u_n \weak u$, $v_n \weak v$ weakly in $H^1(\S^{N-1})$, with strong convergence in $L^2(\S^{N-1})$, and almost everywhere in $\S^{N-1}$. Therefore
\[
\begin{split}
\int_{\S^{N-1}} \langle A \nabla_\theta^A u, \nabla_\theta^A u \rangle = 0, \qquad \int_{\S^{N-1}}|\nabla_\theta v|^2 = 0, \\
\int_{S^{N-1}} u^2 \mu \equiv 1, \qquad \int_{\S^{N-1}} v^2 \equiv 1,  \qquad \int_{\S^{N-1}} u^2 v^2 \equiv 0.
\end{split}
\]
But then it is necessary that $u$ and $v$ are positive constants on $\S^{N-1}$, with disjoint positivity sets, which is clearly a contradiction.

Now we show that $\nu_{A,N} <2$ for $A \neq \textrm{Id}$ as in \eqref{def A}. To this purpose we choose $u=x_1^+$ as test function for $\lambda(A,x_1^+)$:
\begin{align*}
\frac{\int_{\omega_+} \left(  \langle A \nabla x_1, \nabla x_1\rangle - \frac{\langle A \nabla x_1, \nu \rangle^2}{\langle A \nu, \nu \rangle} \right) \,d\sigma   }{\int_{\omega_+} x_1^2 \langle A \nu, \nu \rangle \,d\sigma} 
& =  \frac{1}{\int_{\omega_+} x_1^2 (x_1^2 + \sum_{i >1} a_i x_i^2) \,d\sigma} \int_{\omega_+} \frac{\sum_{i >1} a_i x_i^2}{x_1^2 + \sum_{i >1} a_i x_i^2} \, d\sigma,
\end{align*}
where $\omega_+$ denotes the half-spherical cap $\{x_1>0\} \cap \S^{N-1}$
(recall that $a_1=1$, and $\nu = x$ on $\S^{N-1}$). This proves that
\begin{equation}\label{up bo la 1}
\lambda(A,x_1^+) \le \frac{\int_{\omega_+} \frac{\sum_{i >1} a_i x_i^2}{x_1^2 + \sum_{i >1} a_i x_i^2} \, d\sigma}{\int_{\omega_+} x_1^2 (x_1^2 + \sum_{i >1} a_i x_i^2) \,d\sigma} =: \frac{\varphi(a_2,\dots,a_N)}{\psi(a_2,\dots,a_N)},
\end{equation}
and we aim to show that 
\begin{equation}\label{up bo la}
\lambda(A,x_1^+) < N-1 = \lambda(\textrm{Id}, x_1^+)
\end{equation}
for every matrix $A \neq \textrm{Id}$ as in \eqref{def A}. This amounts to show that 
the right hand side in \eqref{up bo la 1} is strictly smaller than $N-1$ if $A \neq \textrm{Id}$, which in turn is equivalent to prove that
\[
\Phi(a_2,\dots,a_N) := \varphi(a_2,\dots,a_N) - (N-1) \psi(a_2,\dots,a_N) < 0
\]
for every $(a_2,\dots,a_N) \in ([1,+\infty))^{N-1}$, with at least one component $a_j>1$.
Firstly, it is immediate to check that $\Phi(1,\dots,1) =0$. Moreover
\[
\frac{\pa \Phi(a_2,\dots,a_N)}{\pa a_k} = \int_{\omega_+} \frac{x_1^2 x_k^2}{(x_1^2 + \sum_{i>1} a_i x_i^2)^2} \, d\sigma - (N-1) \int_{\omega_+} x_1^2 x_k^2 \, d\sigma < (2-N) \int_{\omega} x_1^2 x_k^2 \,d\sigma \le 0,
\]
where the strict inequality follows from the fact that $a_j>1$ for some $j$. This means that $\Phi$ is monotone decreasing with respect to all its variables in $([1,+\infty))^{N-1}$, and hence $\Phi(a_2,\dots,a_N) < \Phi(1,\dots,1)=0$ for every $(a_2,\dots,a_N) \in ([1,+\infty))^{N-1}$, with at least one component $a_j>1$. Claim \eqref{up bo la} follows.

At this point we proceed with the estimate for $\nu_{A,N}$, by taking the admissible competitor $(u,v) =(x_1^+, x_1^-)$. Since $\lambda(\textrm{Id},x_1^-) = N-1$, by \eqref{up bo la}
\[
\nu_{A,N} \le a_N^{-\frac{N}2} \gamma(\lambda(A,x_1^+)) + \gamma(\lambda(\textrm{Id}, x_1^-)) < \gamma(N-1) + \gamma(N-1) = 2,
\]
which is the desired upper bound.
\end{proof}

We are now in position to proceed with the:

\begin{proof}[Proof of Theorem \ref{thm: ACF}]
We do not stress the depends of the functionals on $u,v$ and $x_0$, to simplify the notation. It is standard to check that $I_A$ and $I_{\textrm{Id}}$ are absolutely continuous functions for $0<r<\rho =\dist(x_0,\pa B_R)$, and hence a.e. $r \in (0,\rho)$ is a Lebesgue point of $J$. Moreover, for a.e. $r \in (0,\rho)$ the restrictions $u|_{S_r(x_0)}, \pa_{x_i} u |_{S_r(x_0)}$ and $v|_{S_r(x_0)}, \pa_{x_i} v |_{S_r(x_0)}$ are functions in $L^2(S_r(x_0))$. We compute the derivative of $J$ with respect to the radius, denoted by $J'$, in any point $r$ for which both the above properties are satisfied, and verify that $J'(r) \ge 0$. 

We suppose that both $I_A(r)>0$ and $I_{\textrm{Id}}(r)>0$, otherwise the fact that $J'(r) \ge 0$ follows simply from the non-negativity of $J$. 

Let $u_{x_0,r}(\cdot)=u(x_0 + r \, \cdot)$, and $v_{x_0,r}(\cdot)=v(x_0 + r \, \cdot)$. By assumption $\int_{\S^{N-1}} u_{x_0,r}^2 v_{x_0,r}^2 = 0$ and, by Lemma \ref{lem: ACF pre} (see also Remark \ref{rmk acf pre}), we have that 
\begin{align*}
\frac{J'(r)}{J(r)} &=  \frac{I_A'(r)}{I_A(r)} + \frac{I_{\textrm{Id}}'(r)}{I_{\textrm{Id}}(r)} -\frac{2\nu_{A,N}}{r} \\
& =  \frac{ \int_{S_r(x_0)} \langle A\nabla u, \nabla u\rangle \Gamma_A}{\int_{B_r(x_0)} \langle A\nabla u, \nabla u\rangle \Gamma_A} +\frac{ \int_{S_r(x_0)} |\nabla v|^2 |x|^{2-N}}{\int_{B_r(x_0)} |\nabla v|^2 |x|^{2-N}}  -\frac{2\nu_{A,N}}{r}\\
& \ge \frac2r \left( a_N^{-\frac{N}2} \gamma(\lambda(A,u_{x_0,r})) +\gamma(\lambda(\textrm{Id},v_{x_0,r}))-\nu_{A,N}\right) \ge 0, 
\end{align*}
where the last inequality follows from the very same definition of $\nu_{A,N}$.
\end{proof}

\subsection{Monotonicity formula in dimension $N=2$}

The $2$-dimensional case is easier than the higher dimensional one, since it is not necessary to work with the fundamental solution $\Gamma_A$. As a consequence, the optimal partition problem defining the exponent in the monotonicity formula is slightly different.

In dimension $N=2$ we modify the definition of $I_A$ as
\begin{equation}\label{def IA 2}
I_A(u,x_0,r) = \int_{B_r(x_0)}  \langle A \nabla u, \nabla u \rangle\,dx.
\end{equation}
As a consequence, Lemma \ref{lem: ACF pre} is simplified as follows.

\begin{lemma}\label{lem: ACF pre 2}
Let $N = 2$, and let $u \in H^1_{\loc}(B_R)$ be nonnegative, and such that $\div(A \nabla u) \ge 0$ in $B_R$. Then, for almost every $r>0$ such that $B_r(x_0) \subset \subset B_R$, we have
\[
I_A(u,x_0,r) \le  \frac{r}{2 \sqrt{\lambda(A,u_{x_0,r})}} \int_{S_r(x_0)} \langle A\nabla u, \nabla u \rangle\, d\sigma.
\]
\end{lemma}

\begin{proof}
Let $x_0=0$ to ease the notation. We test the inequality for $u$ against $u$, and integrate by parts:
\begin{align*}
\int_{B_r} \langle A \nabla u, \nabla u \rangle &\le \int_{B_r} \div(u A \nabla u) = \int_{S_r} u \langle A \nabla u,\nu \rangle \\
&  \le \frac{\sqrt{\lambda(A,u_r)}}{2r} \int_{S_r} u^2 \mu + \frac{r}{2\sqrt{\lambda(A,u_r)}} \int_{S_r} \frac{ \langle A \nabla u,\nu \rangle^2}{\mu} \\
& \le  \frac{r}{2\sqrt{\lambda(A,u_r)}} \left[  \int_{S_r} \langle A \nabla_\theta^A u, \nabla_\theta^A u \rangle + \int_{S_r} \frac{\langle A \nabla u,\nu \rangle^2}{\mu} \right]  = \frac{r}{2\sqrt{\lambda(A,u_r)}} \int_{S_r} \langle A \nabla u, \nabla u \rangle,
\end{align*}
which is precisely  the desired inequality.
\end{proof}

We slightly modify the definition of $\nu_{A,2}$ according to the previous lemma.
\begin{equation}\label{opp2}
\nu_{A,2}:=\inf\left\{ \sqrt{\lambda(A,u)} + \sqrt{\lambda(\textrm{Id},v)}\left| \begin{array}{l} \ u,v \in H^1(\S^{N-1})  \\  
\int_{\S^{N-1}} u^2 v^2 = 0.
\end{array}\right.\right\}.
\end{equation} 
Exactly as in Lemma \ref{lem: on ov}, it is possible to show that $0<\nu_{A,2}<2$ whenever $A \neq \textrm{Id}$. At this point one can proceed as in the higher dimensional case, and prove Theorem \ref{thm: ACF} with the value $\nu_{A,2}$.
%
%
%
%

\subsection{Perturbed monotonicity formula} In this subsection we generalize the previous monotonicity formulae in order to deal with non-segregated subsolutions of a class of elliptic systems. In the symmetric case $A=\textrm{Id}$, this kind of result is obtained in \cite{CTV05, NTTV10, ST15}. We focus only on $N \ge 3$ (as already observed, the case $N = 2$ is a bit simpler) and consider systems of two inequalities such as
\begin{equation}\label{subsys}
\begin{cases}
-\div(A \nabla u) + u^q g_1(x,v) \le 0  \\
-\Delta v + v^p g_2(x,u) \le 0  \\
u,v \ge 0 
\end{cases} \qquad \text{in $\R^N$, with $p,q \ge 1$,}
\end{equation}
under the following assumptions on the continuous functions $g_1, g_2: \R^N \times [0,+\infty) \to [0,+\infty)$: 
\begin{itemize}
\item[(H1)] $\bar g_i(t) := \inf_{x \in \R^N} g_i(x,t)$ is a continuous function of $t \ge 0$, with the property that $\bar g_i(t) >0$ for any $t>0$, and $\bar g_i(t) =0$. Even more, we suppose that $g_i(x,0) = 0$ for every $x \in \R^N$.
\item[(H2)] For every $x \in \R^N$, $g_i(x,\,\cdot)$ is monotone non-decreasing on $[0,+\infty)$.
\end{itemize}
A prototypical example is
\[
g_i(x,t) = \sum_{j=1}^m b_j(x) t^{p_j}, \quad \text{with $\inf_{\R^N} b_j>0$ and $p_j >0$}.
\]

For $(u,v)$ solving \eqref{subsys}, $x_0 \in \R^N$ and $r>0$, we use the following notation
\begin{equation}\label{def I_i}
\begin{split}
I_1(u,v,x_0,r) &= \int_{B_r(x_0)} \left( \langle A \nabla u, \nabla u \rangle + u^{q+1}g_1(x,v) \right) \Gamma_A(x-x_0)\,dx,\\
I_2(u,v,x_0,r) &= \int_{B_r(x_0)} \left( |\nabla v|^2 + v^{p+1}g_2(x,u) \right) |x-x_0|^{2-N}\,dx.
\end{split}
\end{equation}
\begin{theorem}[Perturbed montonicity formula]\label{thm: acf per}
Let $(u,v) \in H^1_{\loc}(\R^N) \cap C(\R^N)$ satisfy \eqref{subsys}, with $g_1,g_2:\R^N \times [0,+\infty) \to \R$ continuous and satisfying (H1) and (H2). For any $\eps>0$ there exist $x_0 \in \R^N$ and $\bar r = \bar r(u,v,\eps)>0$ such that the function
\[
r \mapsto J(u,v, x_0,r)  = \frac{1}{r^{2(\nu_{A,N}-\eps)}} I_1(u,v,x_0,r) I_2(u,v,x_0,r)
\]
is monotone non-decreasing for $r>\bar r$.
\end{theorem}

For the proof, we start with an estimate similar to the one in Lemma \ref{lem: ACF pre}. We introduce
\[
\begin{split}
\Lambda_1(x_0,r) &= \frac{r^2 \int_{S_r(x_0)} \left(\langle A \nabla_\theta^A u , \nabla_\theta^A u \rangle + u^{q+1} g_1(x,v)\right)\,d\sigma}{ \int_{S_r(x_0)} u^2 \mu\, d\sigma} \\
\Lambda_2(x_0,r) & = \frac{r^2 \int_{S_r(x_0)} \left(|\nabla_\theta v|^2 + v^{p+1} g_2(x,u)\right)\,d\sigma}{ \int_{S_r(x_0)} v^2\,d\sigma}.
\end{split}
\]
\begin{lemma}\label{lem: pre ACF per}
In the above setting, for every $x_0 \in \R^N$ and $r>0$ 
\[
I_1(u,v,x_0,r) \le \frac{a_N^{\frac{N}2}r}{2 \gamma(\Lambda_1(x_0,r))} \int_{S_r(x_0)} \left(\langle A \nabla u , \nabla u \rangle + u^{q+1} g_1(x,v)\right) \Gamma_A(x-x_0)\,d\sigma_x.
\]
\end{lemma}
\begin{proof}
Without loss, we consider $x_0=0$, and omit the dependence on $A$ of all the quantities. 
Let $r>0$ be such that ($i$) in the proof of Lemma \ref{lem: ACF pre} holds; almost every $r \in (0,R)$ is admissible. Recalling the definition of $\Phi_{A,\delta}$ (see \eqref{reg_fun}), we take $\delta>0$ such that $\{|A^{-\frac12} x| \le \delta\} \subset \subset B_r$.
By multiplying the inequality for $u$ with $u \Phi_{A,\delta}$, and proceeding as in Lemma \ref{lem: ACF pre}, we obtain
\begin{align*}
\int_{B_r}  \left(\langle A \nabla u , \nabla u \rangle + u^{q+1})\right) \Phi_{\delta} &\le \int_{S_r} \Phi_{\delta} u \langle A \nabla u,\nu \rangle -\frac12\int_{B_r}  \langle  \nabla (u^2), A \nabla \Phi_{\delta} \rangle \\
& \le \int_{S_r} \left( \Gamma u \langle A \nabla u,\nu\rangle -  \frac{u^2}2 \langle A \nabla \Gamma,\nu \rangle \right).  
\end{align*}
By taking the limits as $\delta \to 0^+$, we infer that
\[
\int_{B_r} \left(\langle A \nabla u , \nabla u \rangle + u^{q+1} g_1(x,v)\right) \Gamma \le \int_{S_r} \left( \Gamma u \langle A \nabla u,\nu\rangle -  \frac{u^2}2 \langle A \nabla \Gamma,\nu \rangle \right).
\]
At this point we proceed exactly as in Lemma \ref{lem: ACF pre}, simply replacing $\lambda(A,u_r)$ with $\Lambda_1(0,r)$.
\end{proof}

We also need a suitable variant of the mean value inequality for $A$-subharmonic functions.

\begin{lemma}\label{lem: mvi}
Let $u \in C(\R^N) \cap H^1_{\loc}(\R^N)$ be a nonnegative function such that $\div(A \nabla u) \ge 0$ in $\R^N$. Then there exists $C>0$ depending on $A$ and $N$ such that
\[
\frac1{r^{N-1}} \int_{S_r} u^2 \mu \ge C u^2(0), 
\]
for almost every $r>0$.
\end{lemma}

\begin{proof}
Let $\tilde u(x) = u(A^\frac12 x)$; we have $\Delta(\tilde u^2) \ge 0$ 
in $\R^N$, and then the mean value inequality yields
\[
u^2(0) = \tilde u^2(0) \le \frac{a_N^\frac{N}2}{|B_1| r^N} \int_{B_{a_N^{-1/2}r}} \tilde u^2 = \frac{a_N^\frac{N}2\det A^{-\frac12}}{|B_1| r^N} \int_{\{|A^{-1/2}x | <a_N^{-1/2} r\}} u^2,
\]
for every $r>0$. Now the ellipsoid $\{|A^{-1/2}x | <a_N^{-1/2} r\}$ is contained in the ball $B_r$, so that
\begin{equation}\label{mvib}
u^2(0) \le \frac{a_N^\frac{N}2 \det A^{-\frac12}}{|B_1| r^N} \int_{B_r} u^2, \qquad \forall r>0.
\end{equation}

In order to obtain a similar estimate for the boundary integral, we observe that
\begin{equation}\label{se1}
\int_{B_r} \div(A \nabla (u^2)) \left(r^2-|x|^2 \right) = 2\int_{B_r} \left( u \div(A \nabla u) + \langle A \nabla u, \nabla u \rangle\right)  \left(r^2-|x|^2 \right) \ge 0.
\end{equation}
On the other hand
\begin{equation*}
\begin{split}
\int_{B_r} \div(A \nabla (u^2)) \left(r^2-|x|^2 \right) & = 2\int_{B_r} \langle A \nabla (u^2), x \rangle  = 2 \int_{B_r} \langle \nabla (u^2) , A x \rangle \\
& = 2 r \int_{S_r} u^2 \mu - 2 \sum_i a_i \int_{B_r} u^2,
\end{split}
\end{equation*}
for almost every $r>0$. The thesis follows directly from \eqref{mvib} and \eqref{se1}.
%
\end{proof}

\begin{proof}[Proof of Theorem \ref{thm: acf per}] The proof is similar to the one of \cite[Lemma 5.2]{ST15} (see also \cite[Lemma 7.3]{CTV05}, \cite[Lemma 2.5]{NTTV10}). If $u \cdot  v \equiv 0$ in $\R^N$, then we directly apply Theorem \ref{thm: ACF}. Thus, we can suppose that there exists $x_0 \in \R^N$ with both $u(x_0)>0$ and $v(x_0)> 0$. Without loss, we suppose that $x_0=0$. By continuity, we deduce that $u \cdot v >0$ in a neighborhood of $x_0$, and hence both $I_1(u,v,x_0,r) \neq 0$ and $I_2(u,v,x_0,r) \neq 0$ for every $r>0$. Let now $r>0$ be such that $u|_{S_r}$ and $\pa_{x_i}u|_{S_r}$ are in $L^2(S_r)$, and assume moreover $r$ is a Lebesgue point for $J$; almost every $r>0$ is admissible. As in the proof of Theorem \ref{thm: ACF}, thanks to Lemma \ref{lem: pre ACF per}, we have
\[
\frac{J'(r)}{J(r)}\ge \frac2{r}\left(a_N^{-\frac{N}2} \gamma(\Lambda_1(0,r)) + \gamma(\Lambda_2(0,r))-(\nu_{A,N} - \eps) \right),
\]
and the thesis follows if we show that the right hand side is non-negative for $r$ sufficiently large. Suppose by contradiction that this is not true: then there exists $r_n \to +\infty$ such that 
\begin{equation}\label{cfp}
a_N^{-\frac{N}2} \gamma(\Lambda_1(0,r)) + \gamma(\Lambda_2(0,r)) < \nu_{A,N} - \eps,
\end{equation}
and, in particular, $\{\Lambda_i(r_n)\}$ ($i=1,2$) are bounded sequences. Let
\[
u_{n}(x) = \frac{u(r_n x)}{\left( \frac1{r_n^{N-1}} \int_{S_{r_n}} u^2 \mu\right)^\frac12}, \quad v_n(x)  = \frac{v(r_n x)}{\left( \frac1{r_n^{N-1}} \int_{S_{r_n}} v^2\right)^\frac12}.
\]
We have that 
\begin{align*}
\int_{S_{1}} \langle A \nabla_\theta^A u_n, \nabla_\theta^A u_n \rangle \le \Lambda_1(0,r_n) \qquad \int_{S_{1}} |\nabla_\theta v_n|^2 \le \Lambda_2(0,r_n), \\
\end{align*}
so that $\{u_n\}$ and $\{v_n\}$ are bounded in $H^1(S_1)$, and moreover
\begin{align*}
\int_{S_{1}} u_n^{q+1} \bar g_1\left( \left( \frac1{r_n^{N-1}} \int_{S_{r_n}} v^2\right)^\frac12 v_n   \right) &\le \frac{r_n^2 \int_{S_{r_n}} u^{q+1} g_1(x,v) }{\int_{S_{r_n}} u^2 \mu} \cdot \frac{1}{r_n^2 \left( \frac1{r_n^{N-1}}\int_{S_{r_n}} u^2 \mu \right)^{\frac{q-1}2}} \\
&  \le \frac{\Lambda_1(0,r_n)}{r_n^2 \left( \frac1{r_n^{N-1}}\int_{S_{r_n}} u^2 \mu \right)^{\frac{q-1}2}} \to 0
\end{align*}
as $n \to \infty$, where we used assumption (H1) and Lemma \ref{lem: mvi}. Therefore, we deduce that up to a subsequence $(u_n, v_n) \weak (\tilde u, \tilde v)$ weakly in $H^1(S_1)$, strongly in $L^2(S_1)$, and almost everywhere, where $\tilde u \cdot \tilde v \equiv 0$ on $S_1$: indeed, since $v$ is subharmonic with $v(0)>0$,
\[
\left(\frac1{r^{N-1}} \int_{S_{r_n}} v^2\right)^\frac12 \ge C v(0) =: \delta>0,
\]
and hence by the Fatou lemma and the assumptions on $g_1$
\[
\begin{split}
\int_{S_1} \tilde u^{q+1} \bar g_1(\delta \tilde v) & \le \liminf_{n \to \infty} \int_{S_1} u_n^{q+1} \bar g_1(\delta v_n) \le \liminf_{n \to \infty} \int_{S_{1}} u_n^{q+1} \bar g_1\left( \left( \frac1{r_n^{N-1}} \int_{S_{r_n}} v^2\right)^\frac12 v_n   \right) = 0,
\end{split}
\]
so that in each point of $S_1$ one between $\tilde u$ and $\tilde v$ must vanish, that is $\tilde u \cdot \tilde v \equiv 0$ on $S_1$. 

Coming back to \eqref{cfp}, we obtain by definitions of $\nu_{A,N}$ and $\gamma$ 
\[
\begin{split}
\nu_{A,N} &\le a_N^{-\frac{N}2}\gamma(\lambda(A, \tilde u)) + \gamma(\lambda(\textrm{Id}, \tilde v)) \\
& \le \liminf_{n \to \infty} \left(a_N^{-\frac{N}2} \gamma \left( \int_{S_{1}} \langle A \nabla_\theta^A u_n, \nabla_\theta^A u_n \right) + \gamma \left(\int_{S_{1}}  |\nabla_\theta v_n|^2 \right)\right) \\
& \le \liminf_{n \to \infty} \left( a_N^{-\frac{N}2} \gamma(\Lambda_1(0,r_n)) + \gamma(\Lambda_2(0,r_n)) \right) \le \nu_{A,N}-\eps,
\end{split}
\]
which is a contradiction.
\end{proof}

\section{Liouville-type theorems}\label{sec: liou}

In analogy with the symmetric case $A_i= \textrm{Id}$, the validity of an ACF monotonicity formula allows us to obtain some nonexistence results, both for disjointly supported subsolutions of different linear equations, and for solutions of certain elliptic systems. 

\subsection{Liouville theorem for disjointly supported functions} In this framework, our main achievement is the following.

\begin{theorem}\label{thm: liou seg}
Let $k, N \ge 2$ be positive integers, and, for $i=1,\dots,k$, let $u_i \in H^1_{\loc}(\R^N) \cap C(\R^N)$ be nonnegative functions such that
\[
u_i \cdot u_j \equiv 0 \quad \text{in $\R^N$ if $i \neq j$}, \quad \text{and} \quad -\div(A_i \nabla u_i) \le 0 \quad \text{in $\R^N$},
\]
where $A_i$ are positive definite symmetric matrixes with constant coefficients. There exists an exponent $\bar \nu \in (0,2)$ depending on $N$ and on $A_1,\dots,A_k$ such that the following hold: suppose that for every $i=1,\dots,k$ the functions $u_i$ grow at most like $|x|^{\alpha_i}$, namely 
\[
|u_i(x)| \le C(1+|x|^{\alpha_i}) \qquad \text{for every $|x| \in \R^N$, for some $C>0$},
\]
with 
\begin{equation}\label{cond alpha}
\alpha_i >0 \quad \text{for every $i$, and} \quad  \alpha_i + \alpha_j < \bar \nu \quad \text{for every $i \neq j$};
\eeq
then $k-1$ functions $u_i$ are identically $0$.
\end{theorem} 

In particular, condition \eqref{cond alpha} is satisfied if $\alpha_i =\alpha \in \left(0, \frac{\bar \nu}2\right)$ for every $i=1,\dots,k$, which gives the counterpart of \cite[Proposition 7.2]{CTV05} in the anisotropic framework.

\begin{remark}
We will prove the theorems with a value of $\bar \nu$ explicitly given in terms of a finite number of optimal partitions problems of type \eqref{opp}. In particular, if $k=2$, $A_1=A$ and $A_2 = \textrm{Id}$, then $\bar \nu= \nu_{A,N}$. 
\end{remark}

\begin{remark}\label{rem: opt}
Once that Theorem \ref{thm: liou seg} is proved, it is possible to introduce the optimal exponent for the Liouville theorem in the following way. First, given $k \ge 2$ and positive definite symmetric matrices $A_1,\dots,A_k$, we define
\[
\mathcal{S}_{\nu,N}:= \left\{(u_1,\dots,u_k) \in H^1_{\loc}(\R^N) \cap C(\R^N) \left| \begin{array}  {l} \text{$(u_1,\dots,u_k)$ satisfies all the assumptions of }\\
\text{Theorem \ref{thm: liou seg}, $u_i$ grows at most like $|x|^{\alpha_i}$} \\
\text{with $\alpha_i + \alpha_j < \nu$ for every $i \neq j$,}\\
\text{and at least two components are non-trivial} \end{array}\right.\right\},
\]
and then we set
\begin{equation}\label{op li}
\nu_{\textrm{Liou},N}:= \inf \left\{ \nu >0: \ \mathcal{S}_{\nu,N} \neq \emptyset\right\}. 
\end{equation}
Theorem \ref{thm: liou seg} implies that $\nu_{\textrm{Liou},N} \ge \bar \nu$.

When $A_i = \textrm{Id}$ for every $i$, it follows from \cite[Proposition 7.2]{CTV05} that $\nu_{\textrm{Liou},N} \ge 2$; moreover, $\nu_{\textrm{Liou},N} \le 2$ since $(x_1^+, x_1^-,0,\dots,0)$ is a pair of nontrivial Lipschitz subharmonic functions with disjoint positivity set. Hence in the isotropic case $A_i= \textrm{Id}$ there is a perfect matching between the optimal threshold in the Liouville theorem and the optimal exponent in the ACF monotonicity formula: both of them are equal to $2$. 

In the asymmetric case, it is an open problem to establish whether the equality holds, or if it possible that $\nu_{\textrm{Liou},N} > \bar \nu$, at least for some choices of $A_i$ and $N$. In particular, even if $\bar \nu <2$ by Lemma \ref{lem: on ov}, this does not imply that $\nu_{\textrm{Liou,N}}<2$ as well. However, we shall directly prove that in general $\nu_{\textrm{Liou},N} <2$. To this purpose we construct two non-trivial homogeneous functions $u,v$ of degree $\alpha_1$ and $\alpha_2$, with $\alpha_1 + \alpha_2<2$, satisfying all the assumptions of Theorem \ref{thm: liou seg} for different matrixes $A_1$ and $A_2=\textrm{Id}$. In dimension $N \ge 3$, the existence of such functions was already pointed out in \cite[page 479]{CDS18}, and it is possible to take $u$ and $v$ with the same degree of homogeneity; instead, in dimension $N=2$, it is necessary that the degrees are different. We shall present the examples in details in Subsection \ref{sub: ex} below. 

The examples are relevant since, if Theorem \ref{thm: ACF} were valid with $\nu_{A,N}$ replaced by $2$, then we would been able to prove non-existence as in Theorem \ref{thm: liou seg} for all $\alpha_i+\alpha_j <2$, deducing that $\nu_{\textrm{Liou},N} \ge 2$. But, as discussed above, this is not true. Therefore, the fact that the optimal exponent in in Theorem \ref{thm: ACF} is smaller than $2$ is a natural peculiarity of the anisotropic case, and not a limit of our proof. 
\end{remark}

%

\begin{proof}
Let us consider the pair $u_1$ and $u_2$. Since $A_2$ is positive definite and symmetric, there exist an orthogonal matrix $O$ and a diagonal positive definite matrix $D$ such that $O^t A_2 O = D$. By defining
$\bar u_i(x) = u_i(O D^\frac12 x)$,
it is not difficult to check that $\bar u_1$ and $\bar u_2$ grow at most like $|x|^{\alpha_1}$ and $|x|^{\alpha_2}$ at infinity, respectively, that $\bar u_1 \cdot \bar u_2 \equiv 0$, and that
\[
-\div(\bar A_1 \nabla \bar u_1) \le 0 \quad \text{and} \quad -\Delta \bar u_2 \le 0 
 \quad \text{in $\R^N$},
\] 
where $\bar A_1$ is, again, a positive definite symmetric matrix. Since $\bar A_1$ is positive definite and symmetric, there exist an orthogonal matrix $M$ and a diagonal positive definite matrix $\hat A_1$ such that $M^t \bar A_1 M = \hat A_1$. Without loss of generality, we can suppose that the diagonal elements of $\hat A_1$ are located in increasing order on the diagonal, so that $\hat a := (\hat A_1)_{11}$ is the lowest eigenvalue of $\hat A_1$. 
Let now $u(x) = \bar u_1(\hat a^\frac12 M x)$, $v(x) = \bar u_2(\hat a^\frac12 M x)$, and $A= (\hat a^{-1}) \hat A_1$. Then $u$ and $v$ grow at most like $|x|^{\alpha_1}$ and $|x|^{\alpha_2}$ at infinity, respectively; $u \cdot v \equiv 0$, and
\begin{equation}\label{531}
-\div(A \nabla u) \le 0 \quad \text{and}  \quad -\Delta v \le 0 
 \quad \text{in $\R^N$},
\end{equation}
where $A$ is a diagonal matrix as in \eqref{def A}. In particular, we notice that it is well defined the value $\nu_{12} := \nu_{A,N} \in (0,2)$ given by the optimal partition problem \eqref{opp}. 

The above procedure can be carried out for any pair $(u_i, u_j)$ with $i \neq j$ (actually, by construction $\nu_{ij}=\nu_{ji}$), obtaining a finite number of ACF exponents $\nu_{ij} \in (0,2)$. We take
\begin{equation}\label{nu bar}
\bar \nu:= \min \left\{\nu_{ij}: i \neq j\right\},
\end{equation}
and prove the theorem for this exact choice of $\bar \nu$.

Let us suppose by contradiction that two components, say $u_1$ and $u_2$, are both nontrivial and satisfy all the assumptions of the theorem. The previous argument shows that there exist two nontrivial nonnegative functions $u,v \in H^1_{\loc}(\R^N) \cap C(\R^N)$ with disjoint positivity sets, growing at most like $|x|^{\alpha_1}$ and $|x|^{\alpha_2}$ respectively,
satisfying \eqref{531}. Notice that $\alpha_1+\alpha_2 <  \nu_{12} = \nu_{A,N}$. By using the asymmetric monotonicity formula we show that this provides a contradiction, following the same strategy originally developed in \cite[Proposition 7.2]{CTV05}. 

The segregation condition $u \cdot v \equiv 0$ implies that there exists $x_0 \in \R^N$ and $\bar r>0$ sufficiently large such that $u(x_0) = v(x_0) = 0$, and both $u$ and $v$ are non-constant in $B_{\bar r}(x_0)$. In particular, $J(u,v,x_0,\bar r) >0$, so that, by the monotonicity formula in Theorem \ref{thm: ACF},
\begin{equation}\label{below}
I_A(u,x_0,r) I_{\textrm{Id}}(v,x_0,r) \ge C r^{2\nu_{A,N}} \qquad \text{for $r>\bar r$}.
\end{equation}
Let now $r> \bar r$, and consider a radial smooth cutoff function $\eta$ such that $0 \le \eta \le 1$, $\eta =1$ in $B_r(x_0)$, $\eta = 0$ in $\R^N \setminus B_{2r}(x_0)$, and $|\nabla \eta| \le C/r$. Let also $\delta >0$ be such that $\{|A^\frac12 x| \le \delta \} \subset \subset B_r$. By testing the inequality satisfied by $u$ with $\eta^2 \Phi_{A,\delta}(x-x_0) u$ (with $\Phi_{A,\delta}$ defined in \eqref{reg_fun}), we obtain
\begin{equation}\label{2631}
\begin{split}
\int_{B_{2r}(x_0)} & \eta^2 \Phi_{A,\delta}(x-x_0) \langle A \nabla u, \nabla u \rangle \\
&\le - \int_{B_{2r}(x_0)} \left( 2 \eta u  \Phi_{A,\delta}(x-x_0) \langle A \nabla u, \nabla \eta \rangle + \eta^2 u \langle A \nabla u, \nabla  \Phi_{A,\delta}(x-x_0) \right) \\
& \le \int_{B_{2r}(x_0)} \frac12 \eta^2  \Phi_{A,\delta}(x-x_0) \langle A \nabla u, \nabla u \rangle + 2  \Phi_{A,\delta}(x-x_0) u^2 \langle A \nabla \eta, \nabla \eta \rangle\\
& \qquad  - \int_{B_{2r}(x_0)}  \eta^2 u \langle A \nabla u, \nabla  \Phi_{A,\delta}(x-x_0)\rangle.
\end{split}
\end{equation}
In order to deal with the last term, we recall that $\div(A \nabla \Phi_{A,\delta}) \le 0$ in $\R^N$, whence it follows that
\begin{align*}
0 &\le \int_{B_{2r}(x_0)} \frac12 \langle A \nabla   \Phi_{A,\delta}(x-x_0), \nabla (u^2 \eta^2) \rangle \\
& = \int_{B_{2r}(x_0)} \left(\eta u^2 \langle A \nabla  \Phi_{A,\delta}(x-x_0), \nabla \eta \rangle + \eta^2 u \langle A \nabla u, \nabla  \Phi_{A,\delta}(x-x_0) \rangle \right).
\end{align*}
Hence \eqref{2631} yields
\begin{align*}
\int_{B_{r}(x_0)} & \Phi_{A,\delta}(x-x_0) \langle A \nabla u, \nabla u \rangle \\
& \le \int_{B_{2r}(x_0) \setminus B_r(x_0)} \left(4 \Gamma_A(x-x_0) u^2 \langle A \nabla \eta, \nabla \eta \rangle +2 \eta u^2 \langle A \nabla \eta, \nabla \Gamma_A(x-x_0)\rangle\right),
\end{align*}
where we used the fact that $\nabla \eta \equiv 0$ in $B_r(x_0)$, and $\Gamma_A = \Phi_{A,\delta}$ outside $B_r(x_0)$. By taking the limit as $\delta \to 0^+$, thanks to the Fatou lemma and the growth condition on $u$, we infer that 
\[
\begin{split}
I_A(u,x_0,r) & \le \frac{C}{r^2} \int_r^{2r} \frac{\rho^{2\alpha_1}}{\rho^{N-2}} \rho^{N-1}\,d\rho + \frac{1}{r} \int_r^{2r} \frac{\rho^{2\alpha_1}}{\rho^{N-1}} \rho^{N-1}\,d\rho \le C r^{2\alpha_1}.
\end{split}
\]
In the same way, by testing the inequality satisfied by $v$ with $\eta^2 \Phi_{\textrm{Id},\delta}(x-x_0) v$, one can show that 
\[
\begin{split}
I_{\textrm{Id}}(v,x_0,r) & \le C r^{2\alpha_2}.
\end{split}
\]
By combining the former inequalities with \eqref{below}, we finally conclude that for $r>\bar r$
\[
C_1 r^{2\nu_{A,N}} \le I_A(u,x_0,r) I_{\textrm{Id}}(v,x_0,r) \le C_2 r^{2(\alpha_1+\alpha_2)}, 
\]
which is a contradiction for large $r$ since $\alpha_1+\alpha_2 < \nu_{A,N}$.
\end{proof}

%
%
%
%
%


\subsection{Liouville theorem for subsolutions and solutions to certain elliptic systems}\label{sub: liou} 
Our first goal is to prove nonexistence of nontrivial nonnegative subsolutions for a system with $2$ components. 

\begin{theorem}\label{thm: liou sub}
Let $N \ge 2$, and let $u,v \in H^1_{\loc}(\R^N) \cap C(\R^N)$ satisfy \eqref{subsys} in $\R^N$, under assumptions (H1) and (H2) on the coupling terms $g_1$ and $g_2$. Assume moreover that for some $u$ and $v$ grow at most like $|x|^{\alpha_1}$ and $|x|^{\alpha_2}$, respectively, with
\[
\alpha_1,\alpha_2 >0, \quad \text{and} \quad \alpha_1+\alpha_2 < \nu_{A,N};
\]
then at least one between $u$ and $v$ vanishes identically.
\end{theorem}

\begin{proof}
Suppose by contradiction that neither $u$ nor $v$ vanishes identically. Let $0<\eps<\nu_{A,N}-(\alpha_1+\alpha_2)$. Then, by Theorem \ref{thm: acf per}, there exist $x_0 \in \R^N$ and $C, \bar r>0$ such that 
\begin{equation}\label{below'}
I_1(u,v, x_0,r) I_2(u,v,x_0,r) \ge Cr^{2(\nu_{A,N}-\eps)}
\end{equation}
for $r> \bar r$. On the other hand, let $\eta$ be a cutoff function as in the proof of Theorem \ref{thm: liou seg}, and let $\Phi_{A,\delta}$ be defined in \eqref{reg_fun}, with $\delta>0$ such that $\{|A^\frac12 x|<\delta\}\subset B_{\bar r}$. By testing the inequality satisfied by $u$ (resp. $v$) by $\eta^2 \Phi_{A,\delta}(x-x_0) u$ (resp $\eta^2 \Phi_{\delta}(x-x_0) v)$, we obtain 
\begin{align*}
\int_{B_{2r}(x_0)} & \left( \langle A \nabla u, \nabla u \rangle + u^{q+1}g_1(x,v) \right) \Phi_{A,\delta}(x-x_0) \le \int_{B_{2r}(x_0)} \frac12 \eta^2  \Phi_{A,\delta}(x-x_0) \langle A \nabla u, \nabla u \rangle  \\
& + \int_{B_{2r}(x_0)} 2 \Phi_{A,\delta}(x-x_0) u^2 \langle A \nabla \eta, \nabla \eta \rangle -  \eta^2 u \langle A \nabla u, \nabla  \Phi_{A,\delta}(x-x_0)\rangle
\end{align*}
as in the proof of Theorem \ref{thm: liou seg}, it is not difficult to deduce that 
\[
I_1(u,v,x_0,r) \le C r^{2\alpha_1}
\]
for $r> \bar r$. In the same way it is possible to estimate $I_2(u,v,x_0,r)$, obtaining a contradiction with \eqref{below'} since $\alpha_1+\alpha_2 <\nu_{A,N}-\eps$.
\end{proof}

As application, we present a general Liouville theorem for possibly sign-changing solutions of some elliptic systems with arbitrarily many components. To state our results in full generality, we introduce some notation. Let $k, N \ge 2$ be positive integers. For an arbitrary $m \leq k$, we say that a vector $\mf{b}=(b_0,\dots,b_m) \in \N^{m+1}$ is an \emph{$m$-decomposition of $k$} if
\[
0=b_0<b_1<\dots<b_{m-1}<b_m=k;
\]
given a $m$-decomposition $\mf{b}$ of $k$, we set, for $h=1,\dots,k$,
\begin{equation}\label{def indexes}
\begin{split}
& I_h:= \{i \in  \{1,\dots,d\}:  b_{h-1} < i \le b_h \}, \\
&\mathcal{K}_1:= \left\{(i,j) \in I_h^2 \text{ for some $h=1,\dots,m$, with $i \neq j$}\right\}, \\ &\mathcal{K}_2 := \left\{(i,j) \in I_{h_1} \times I_{h_2} \text{ with $h_1 \neq h_2$} \right\}. 
\end{split}
\end{equation}

Let now $\mf{u}=(u_1,\ldots, u_k) \in H^1_{\loc}(\R^N) \cap C(\R^N)$ satisfy
\begin{equation}\label{sys liou}
-\div(A_i \nabla u_i)= - \mathop{\sum_{j=1}^k}_{j\neq i}  u_i |u_i|^{p_{ij}-1} g_{ij}(x,|u_j|) \quad \text{ in } \R^N,\qquad i=1,\ldots, d,\\
\end{equation}
under the following assumptions on the data:
\begin{itemize}
\item[(G1)] $A_i$ are positive definite symmetric matrixes with constant coefficients;
\item[(G2)] $p_{ij}>0$ for every $i \neq j$, and $p_{ij} \ge 1$ for every $(i,j) \in \mathcal{K}_2$;
\item[(G3)] $g_{ij} \equiv 0$ for $(i,j) \in \mathcal{K}_1$, and $g_{ij}$ satisfies assumptions (H1) and (H2) in Theorem \ref{thm: acf per} for every $(i,j) \in \mathcal{K}_2$. 
\end{itemize}

The term $-u_i |u_i|^{p_{ij}-1} g_{ij}(x,|u_j|)$ describes the interaction between $u_i$ and $u_j$. By introducing a $m$-decomposition of $k$, we have divided the components of $\mf{u}$ into $m$ groups: $\{u_i: i \in I_1\}$, \dots, $\{u_i: i \in I_m\}$. Assumption (G3) means that $u_i$ and $u_j$ do not interact ($g_{ij}=0$) if $(i,j) \in \mathcal{K}_1$, i.e. if $u_i$ and $u_j$ are in the same group; instead, they interact in a competitive way ($g_{ij}>0$) if $(i,j) \in \mathcal{K}_2$, i.e. if $u_i$ and $u_j$ are in different groups.

\begin{theorem}\label{thm: liou syst}
In the above setting, let $\bar \nu \in (0,2)$ be given by Theorem \ref{thm: liou seg}.  Suppose that each function $u_i$ grows at most like $|x|^{\alpha_i}$, where
\[
\alpha_i>0 \quad \text{for every $i$, and} \quad \alpha_i + \alpha_j < \bar \nu \quad \text{for every $(i,j) \in \mathcal{K}_2$}.
\]
Then there exists $\ell \in \{1,\dots,m\}$ such that $u_i \equiv 0$ for every $i \in I_h$ with $h \neq \ell$, and $u_i$ is constant for $i \in I_\ell$.
\end{theorem}

\begin{remark}
A similar Liouville theorem was proved in \cite{STTZ16}, for a specific choice of $g_{ij}$. The validity of Theorem \ref{thm: liou syst} allows us to extend the validity of Theorems \ref{thm: lv} and \ref{thm: be} in cases when the competition takes place among groups of components, as in \cite{STTZ16}. We do not insist on this point for the sake of simplicity.
\end{remark}

\begin{proof}
We show that it is possible to apply Theorem \ref{thm: liou sub} to any couple $(u_i,u_j)$ where $i \in I_h$, $j \in I_k$, with $h \neq k$. Then it is necessary that $m-1$ groups of components vanish identically, and the components of the last group are constants (by (G3), they are harmonic and globally H\"older continuous in $\R^N$).

Suppose at first that $u_i$ and $u_j$ are also non-negative. Then
\[
\begin{cases}
-\div(A_i \nabla u_i) \le - u_i^{p_{ij}} g_{ij}(x,u_j) & \text{in $\R^N$} \\
-\div(A_j \nabla u_j) \le - u_j^{p_{ji}} g_{ji}(x,u_i) & \text{in $\R^N$.}
\end{cases}
\]
As in the proof of Theorem \ref{thm: liou seg}, it is possible to suppose that $A_j = \textrm{Id}$ and $A_i =A$ is diagonal as in \eqref{def A}. Thus, it is well defined $\nu_{A,N}$ as in \eqref{opp}, and, recalling the definition \eqref{nu bar} of $\bar \nu$, we have that $2\alpha<\nu_{A,N}$. Therefore, one between $u_i$ and $u_j$ must vanish identically by Theorem \ref{thm: liou sub}. 

If instead the components can change sign, recalling the assumptions on $g_{ij}$ we have that
\[
\begin{cases}
-\div(A_i \nabla u_i^+) \le - (u_i^+)^{p_{ij}} g_{ij}(x,u_j^+) & \text{in $\R^N$} \\
-\div(A_j \nabla u_j^+) \le - (u_j^+)^{p_{ji}} g_{ji}(x,u_i^+) & \text{in $\R^N$,}
\end{cases}
\]
and analogue systems are satisfied by $(u_i^+, u_j^-)$, $(u_i^-, u_j^+)$, $(u_i^-, u_j^-)$. In each case, it is possible to suppose that $A_j = \textrm{Id}$ and $A_i$ is diagonal as in \eqref{def A}. Thus, by applying Theorem \ref{thm: liou sub} to all the possible pairs, we deduce that at least one between $u_i$ and $u_j$ vanishes identically. 
\end{proof}

\subsection{Upper estimate on $\nu_{\textrm{Liou},N}$}\label{sub: ex}

In this section we show that, at least for a suitable choice of $A_1$ and $A_2$, the optimal value $\nu_{\textrm{Liou},N}$ defined in \eqref{op li} is strictly less than $2$. This follows directly from the following:

\begin{proposition}\label{prop: ex}
Let $N \ge 2$. There exists a positive definite diagonal matrix $A$ with constant coefficients, two disjoint open cones $\cC_1, \cC_2$ of $\R^N$, and two non-negative and non-trivial homogeneous functions $u$ and $v$ in $H^1_{\loc}(\R^N) \cap C(\R^N)$, of degree $\alpha_1 >0$ and $\alpha_2>0$, with $\alpha_1 + \alpha_2<2$, such that 
\[
\div(A \nabla u) =0 \quad \text{in $\cC_1=\{u>0\}$}, \quad \Delta v = 0 \quad \text{in $\cC_2=\{v>0\}$}.
\]
Moreover, if $N \ge 3$ we can construct $u$ and $v$ with $\alpha_1=\alpha_2$.
\end{proposition}
%
%
%

\begin{proof}[Proof of Proposition \ref{prop: ex} in dimension $N=2$]
Let $\phi_1, \phi_2 \in (0,\pi/2)$, $\omega_1 = (-\phi_1, \phi_1)$, and $\omega_2= (\phi_2, 2\pi-\phi_2)$. We consider the eigenvalue problems on the circle
\[
\begin{cases}
- \varphi''=\lambda \varphi, \quad \varphi > 0 & \text{in $\omega_1$}\\
\varphi = 0 & \text{on $\pa \omega_1$},  
\end{cases} \qquad \begin{cases}
- \psi''=\mu \psi, \quad \psi > 0 & \text{in $\omega_2$}\\
\psi = 0 & \text{on $\pa \omega_2$}.  
\end{cases}
\]
The problems can be explicitly solved, deducing in particular that $\lambda= \left(\pi/(2 \phi_1)\right)^2$, $\mu = \left( \pi/(2(\pi-\phi_2))\right)^2$. Let $\varphi_1$ and $\psi_1$ denote the corresponding normalized eigenfunctions, and let $\alpha_1 = \sqrt{\lambda}$ and $\alpha_2= \sqrt{\mu}$; it is well known that $w = r^{\alpha_1} \varphi_1$ and $v = r^{\alpha_2} \psi_1$ are homogeneous harmonic functions in the cones $\cD_{1}, \cC_{2}$ generated by $\omega_1$ and $\omega_2$, respectively. Notice that $\alpha_1>1$ can be made arbitrarily close to $1$ by taking $\phi_1$ close to $\pi/2$. Similarly, $\alpha_2<1$ can be made close to $1/2$ by taking $\phi_2$ close to $0$. In particular, for any $0<\eps<1/2$ we can take $\phi_1$ and $\phi_2$ such that
\[
\alpha_1>1, \quad \alpha_2<1, \quad \alpha_1 + \alpha_2 <\frac{3}{2}+\eps<2.
\]
Now, for $b \in (0,1)$, let 
\[
B = \left( \begin{array}{c c} b & 0 \\ 0 & 1 \end{array}\right), \quad A=B^{-1},
\]
and $u(x) = w(B^\frac12 x)$. Then $\div(A \nabla u) =0$ in $\cC_1 = \{x \in \R^2: B^\frac12 x \in \cD_1\} = B^{-1/2} \cD_1$, $u=0$ on $\pa \cC_1$, and it is homogeneous of degree $\lambda$. $\cC_1$ is a cone, generated by a set $\omega' \subset \mathbb{S}^1$, and it is not difficult to check that, if $b$ is sufficiently small, then $\omega' \subset \S^1 \setminus \omega_2$. Therefore $u$ and $v$ provide the desired example.
\end{proof}

\begin{remark}
Notice that, up to exchanging the role of the variables $x_1$ and $x_2$, the matrix $A$ satisfies the structural assumptions \eqref{def A}, i.e. it is a diagonal matrix with lower entry equal to $1$.

It is interesting that in the previous example $u$ is superlinear and $v$ is sublinear. This means in particular that, even if we take $b$ in a such a way that $\pa \omega' = \pa \omega_2$, then $u$ and $v$ cannot satisfy a free-boundary condition of the type
\[
\pa_\nu u=G(\pa_\nu v, \nu) \quad \text{on $\pa \omega'$, with $G$ increasing with respect to its first variable}.
\]
This is in accordance with the main result in \cite{CDS18}, which implies in particular that in dimension $N=2$ one cannot construct an example where $u$ and $v$ have the same degree of homogeneity less than $1$.
\end{remark}

Now we consider the case $N \ge 3$. Of course, the two dimensional example can be considered also in higher dimension. We think however that it is interesting to produce an example where $u$ and $v$ have the same degree (which is not possible in dimension $N=2$). The idea of the construction was suggested to us by Daniela De Silva in a personal communication \cite{DS}. We start with a preliminary result concerning an eigenvalue problem on the unit sphere $\S^2$. We parametrize the sphere with spherical coordinates $(\varphi, \theta) \in [0,\pi] \times [-\pi,\pi]$ ($\varphi$ is the polar angle, $\theta$ is the azimuthal angle).

\begin{lemma}\label{lem: 1 band}
For $\alpha \in (\pi/2,\pi)$, $\beta \in (0,\pi/2)$, let 
\[
\omega=\left\{(\varphi, \theta) \in \left(\frac{\pi}{2}-\beta, \frac{\pi}2 + \beta \right) \times (-\alpha,\alpha)\right\}.
\]
There exist $\alpha$ and $\beta$ such that the first eigenvalue of the problem
\[
\begin{cases}
-\Delta_{\mathbb{S}^2} u = \lambda u & \text{in $\omega$} \\
u = 0 & \text{on $\pa \omega$}
\end{cases}
\]
is strictly smaller than $2$. 
\end{lemma}

\begin{proof}
We separate variables by letting $u(\varphi,\theta) = v(\theta) w(\varphi)$, plug this ansatz in the differential equation for $u$, and search for a positive solution. The differential equation reads
\[
\frac{\sin \varphi(\sin \varphi \, w')'}{w} + \lambda \sin^2 \varphi = -\frac{v''}{v}.
\]
Hence there exists $c \in \R$ such that $v'' +c v =0$, which together with $v>0$ and the boundary conditions $v(-\alpha) = 0 = v(\alpha)$ implies that $c = m^2 = \left( \pi/(2\alpha)\right)^2 $ and $v(\theta) = \cos(m\theta)$ (up to a multiplicative constant). At this point we come back to the boundary value problem for $w$; 
by changing variable $s=\cos \varphi$, we obtain
\begin{equation}\label{sl pb}
\begin{cases}
-( (1-s^2) \bar w')' + \frac{m^2}{1-s^2} \bar w = \lambda \bar w & \text{in }\left(-\rho, \rho\right) \\
\bar w\left(-\rho\right) = 0 = \bar w\left(\rho\right),
\end{cases}
\end{equation}
where $\rho = \cos\left(\frac{\pi}2 - \beta\right) \in (0,1)$ and $\bar w(s) = w(\varphi(s))$. This is a typical Sturm-Liouville problem with strictly positive potential $m^2/(1-s^2)$, and hence the existence of a first positive eigenvalue $\lambda_1$, together with a first positive normalized eigenfunction $w_1$, is guaranteed. We need an upper bound on $\lambda_1$, and this can be obtained from the variational characterization
\[
\lambda_1 = \inf_{\varphi \in H_0^1(-\rho, \rho) \setminus \{0\}} Q_{\rho,m}(\varphi) = \inf_{\varphi \in H_0^1(-\rho, \rho) \setminus \{0\}}\frac{\int_{-\rho}^\rho (1-s^2) (\varphi')^2 + \frac{m^2}{1-s^2} \varphi^2}{\int_{-\rho}^\rho \varphi^2}.
\]
By choosing the test function $\psi(s) = \cos\left(\frac{\pi s}{2\rho}\right)$, we infer that
\[
\lambda_1 \le Q_{\rho,m}(\psi) = \frac{\pi^2}{4\rho^2} \int_{-1}^1 (1-\rho^2 t^2) \sin^2\left(\frac{\pi}2 t\right)\,dt + m^2 \int_{-1}^1 \frac{\cos^2\left( \frac{\pi}2 t\right)}{1-\rho^2 t^2} \, dt.
\]
The right hand side is continuous with respect to $(\rho,m) \in (0,1] \times \R^+$. By taking $\alpha \simeq \pi$ and $\beta \simeq \pi/2$, we can make $m$ arbitrarily close to $1/2$, and $\rho$ arbitrarily close to $1$. This means that for such a choice of $\alpha$ and $\beta$ we have that
\[
\lambda_1 \le Q_{\rho,m}(\psi) \simeq Q_{1,1/2}(\psi) \approx \frac{\pi^2}4 \cdot 0.47 + \frac14 \cdot 1.22 \approx 1.47 <2,
\]
which is the desired result.
\end{proof}

\begin{proof}[Proof of Proposition \ref{prop: ex} in dimension $N \ge 3$]
The main idea is to show the existence of a domain $\omega'$ on the sphere $\S^2$ that contains more than half of a great circle such that, for suitable $\mu \in (0,1)$ and a positive definite symmetric constant matrices $A$, the solution of $\div(A \nabla w)=0$ which vanishes on the cone generated by $\omega'$ has homogeneity $\mu$. If $N=3$ we can take two complementary domains with this property (for instance those separated by the white line of a typical tennis ball).

We now present the details. Let us consider the half great circle $\gamma_1 = \{x \in \mathbb{S}^2: x_3=0, x_2>0\}$, and let $\omega=\{(\varphi, \theta) \in [\pi/2-\beta, \pi/2+\beta] \times [-\delta, \pi+\delta]\}$, where $\delta \in (0,\pi/2)$ is such that $\pi+2\delta = 2\alpha$, with $\alpha$ and $\beta$ given by Lemma \ref{lem: 1 band}. Then the first eigenvalue of the Laplace-Beltrami operator on $\omega$, with homogeneous Dirichlet boundary condition, is smaller than $2$, and this implies that the positive harmonic function $w$ in the cone $\cD_1$ generated by $\omega$, vanishing on $\pa \omega$, has homogeneity $\mu <1$. Now, for $b \in (0,1)$, we consider the diagonal matrices 
\[
B = \left( \begin{array}{c c c}  b^2 & 0 & 0  \\ 0 & b & 0 \\ 0 & 0 & 1 \end{array}\right), \quad A=B^{-1},
\]
and let $u(x) = w(B^\frac12 x)$. Then $\div(A \nabla u) =0$ in $\cC_1 = \{x \in \R^3: B^\frac12 x \in \cD_1\}$, $u=0$ on $\pa \cC_1$, and it is homogeneous of degree $\mu$. $\cC_1$ is a cone, generated by a set $\omega' \subset \mathbb{S}^2$. It is not difficult to check that the set $\omega'$ can be included in an arbitrarily small neighborhood of $\gamma_1$, by taking $b$ sufficiently small. Now we consider a second band $\omega_2$ of the same type of $\omega$, but surrounding the half great circle $\gamma_2 =\{x \in \mathbb{S}^2: x_1 =0, \ x_2<0\}$. We fix $b$ so small that $\omega_2 \cap \omega' = \emptyset$, and notice that by Lemma \ref{lem: 1 band} the positive harmonic function $v$ in the $\cC_2$ generated by $\omega_2$, vanishing on $\pa \omega_2$, is homogeneous of degree $\mu<1$. Thus the pair $(u,v)$ fulfills all the requirement of the theorem (with a matrix $A$ satisfying the structural assumptions \eqref{def A}, up to exchanging the coordinates).
\end{proof}

\section{Spatial segregation of competitive systems: Lotka-Volterra interaction}\label{sec: lv}

In this section we prove Theorems \ref{thm: lv} and \ref{thm: limit lv}, by following the blow-up method used in \cite[Theorem 4]{CTV05}. Before entering the core of the proof, we observe that each $\mf{u}_\beta$ is $C^{1}$ up to the boundary, and $\{\mf{u}_\beta\}$ is uniformly bounded in $L^\infty(\Omega)$, since each $u_{i,\beta}$ is $L_i$-subharmonic and the boundary data are fixed. Notice also that we can define $\bar \nu= \bar \nu(N, A_1,\dots,A_k) \in (0,1)$ as in Theorem \ref{thm: liou seg}.  

\begin{lemma}\label{lem: cpq}
Let $w \in H^1(B_{2r}) \cap C(\overline{B_{2r}})$ be a positive subsolution to 
\[
- \div(A \nabla w) \le - M w + \delta \qquad \text{in $B_{2r}$,}
\]
with $M>0$, $\delta  \ge 0$, and $A$ positive definite, symmetric, with constant coefficients. Then there exist $C, c>0$ such that
\[
\sup_{x \in B_r} w(x) \le C \| w\|_{L^\infty(B_{2r})}  e^{- c r \sqrt{M}}  + \frac{\delta}{M}.
\]  
\end{lemma}

\begin{proof}
Let $x_0 \in B_r$. The function $\bar w:= w/\|w\|_{L^\infty(B_r(x_0))}$ is a positive subsolution to 
\[
- \div(A \nabla \bar w) \le - M \bar w + \frac{\delta}{\|w\|_{L^\infty(B_r(x_0))}}, \qquad 	\bar w \le 1 \qquad \text{in $B_r(x_0)$}.
\]
Let $\Lambda$ be the maximal eigenvalue of $A$. Then, as observed in  \cite[Lemma 5.2]{CPQ}, the function
\[
z(x) = \sum_{i=1}^N \cosh\left(\frac{ \sqrt{M} x_i}{\Lambda}\right)
\]
is a supersolution of $\div(A \nabla z) \le M z$ in $B_r$, and satisfy
\[
z(x) \ge C e^{c \sqrt{M} r}  \qquad \text{for every $x \in S_r$}
\]
for suitable $c, C>0$ depending on $A$ and $N$. Let us consider 
\[
\bar z(x) := \frac{z(x-x_0)}{C e^{c\sqrt{M}r}} + \frac{\delta}{M \|w\|_{L^\infty(B_r(x_0))}}.
\]
We have
\[
- \div(A \nabla \bar z) \ge - M \bar z + \frac{\delta}{\|w\|_{L^\infty(B_r(x_0))}} \qquad \text{in $B_r(x_0)$}, 
\]
with $\bar z \ge 1$ on $S_{r}(x_0)$. Then the comparison principle yields
\[
w(x_0) \le C \|w\|_{L^\infty(B_r(x_0))} z(0) e^{-c \sqrt{M} r} + \frac{\delta}{M} \le C \|w\|_{L^\infty(B_{2r})} e^{-c \sqrt{M} r} + \frac{\delta}{M},
\] 
and we obtain the thesis by taking the supremum over $x_0 \in B_r$.
\end{proof}

\begin{lemma}\label{lem: liou ent}
Let $A$ be a positive definite symmetric matrix with constant coefficients. Suppose that $w$ is globally $\alpha$-H\"older continuous in $\Omega$, for some $\alpha \in (0,1)$. 
\begin{itemize}
\item[($i$)] If $\div(A \nabla w) = 0$ in $\Omega=\R^N$, then $w$ is constant.
\item[($ii$)] If $\div(A \nabla w) = 0$ in a half-space $\Omega$, and $w$ is constant on the boundary, then it is constant.
\end{itemize}
\end{lemma}

Here and in what follows we say that a function $w$ is globally $\alpha$-H\"older continuous in $\Omega$ if its $\alpha$-H\"older semi-norm $[w]_{C^{0,\alpha}(\overline{\Omega})}$ is bounded; notice that we do not ask that $w \in L^\infty(\Omega)$.

\begin{proof}
($i$) After a rotation and a scaling, we obtain a harmonic function $\tilde w$ in $\R^N$, still globally $\alpha$-H\"older continuous, thus constant by the Liouville theorem.

($ii$) After a rotation and a scaling, we obtain a harmonic function $\tilde w$ in a half-space, constant on the boundary of the half-space. We can then extend it in a symmetric way, to obtain a harmonic function in the whole space $\R^N$, still globally $\alpha$-H\"older continuous, and hence constant.
\end{proof}

We address now the proof of Theorem \ref{thm: lv}. Let $\alpha \in (0,\bar \nu/2)$, and suppose by contradiction that $\{\mf{u}_\beta\}$ is not bounded in $C^{0,\alpha}(\overline{\Omega})$, namely there exists a sequence $\beta \to +\infty$ such that 
\[
L_\beta:= \sup_i \sup_{\substack{x \neq y, \\ x, y \in \overline{\Omega}}} \frac{|u_{i,\beta}(x)- u_{i,\beta}(y)|}{|x-y|^\alpha}  \to +\infty.
\]
Since, for each $\beta$ fixed, $\mf{u}_\beta$ is of class $C^{0,\alpha'}(\overline{\Omega})$ with $\alpha' > \alpha$, we can assume w.l.o.g. that $L_\beta$ is achieved by $u_{1,\beta}$ at the pair $(x_\beta, y_\beta)$. The uniform boundedness in $L^\infty(\Omega)$ yields
\[
|x_\beta-y_\beta|^\alpha = \frac{|u_{1,\beta}(x_\beta)- u_{1,\beta}(y_\beta)|}{L_\beta} \le \frac{2 \|u_{1,\beta}\|_{L^\infty(\Omega)}}{L_\beta} \to 0.
\]
We consider the following blow-up of $\mf{u}_\beta$ with center in $x_\beta$, with $r_\beta \to 0^+$ to be chosen later:
\[
\mf{v}_\beta(x):= \frac{1}{L_\beta r_\beta^\alpha} \mf{u}_\beta(x_\beta + r_\beta x), \qquad x \in \Omega_\beta := \frac{\Omega-x_\beta}{r_\beta}.
\]
According to the behavior of $\dist(x_\beta, \pa \Omega_\beta)$, and by the regularity of $\pa \Omega$, either $\Omega_\beta$ exhausts $\R^N$ as $\beta \to +\infty$, or $\Omega_\beta$ tends to a half-space. In both cases, we denote the limit domain by $\Omega_\infty$. 

Plainly, $\mf{v}_\beta$ is a positive solution to
\[
\begin{cases}
L_i v_{i,\beta}= M_\beta v_{i,\beta} \sum_{j \neq i} b_{ij}  v_{j,\beta}  &\text{ in $\Omega_\beta$}\\
v_{i,\beta} = \varphi_{i,\beta} &\text{ on $\pa \Omega_\beta$},
\end{cases} 
\]
where $M_\beta = L_\beta r_\beta^{2+\alpha} \beta$, and $\varphi_{i,\beta}$ is defined by scaling the boundary datum $\varphi_i$. Furthermore, for all $\beta >1$
\[
\max_i \max_{\substack{x \neq y, \\ x, y \in \overline{\Omega}}} \frac{|v_{i,\beta}(x)- v_{i,\beta}(y)|}{|x-y|^\alpha} =  \frac{|v_{1,\beta}(0)- v_{1,\beta}\left(\frac{y_\beta-x_\beta}{r_\beta} \right)|}{\left|\frac{x_\beta-y_\beta}{r_\beta}\right|^\alpha} = 1.
\]

The next lemma will be useful in order to deal with the case when the scaled domains converge to a half-space.

\begin{lemma}\label{lem: half-space}
Suppose that $\Omega_\beta$ tends to a half-space $\Omega_\infty$. Then it is possible to extend $\mf{v}_\beta$ outside $\Omega_\beta$ in a Lipschitz fashion, in such a way that:
\begin{itemize}
\item[($i$)] If $\{\mf{v}_\beta(0)\}$ is bounded, then 
$\mf{v}_\beta \to \mf{v}$ in $C^{0,\alpha'}_{\loc}(\R^N)$ for every $0<\alpha'<\alpha$, up to a subsequence; moreover, the limit function $\mf{v}$ attains a constant value on the boundary $\pa \Omega_\infty$, and at most one component is different from $0$ in $\R^N$.
\item[($ii$)] If $\{\mf{v}_\beta(0)\}$ is unbounded, then $\tilde{\mf{v}}_\beta(x) := \mf{v}_\beta(x)-\mf{v}_\beta(0)$ converges to $\tilde{\mf{v}}$ in $C^{0,\alpha'}_{\loc}(\R^N)$ for every $0<\alpha'<\alpha$, up to a subsequence; moreover, the limit function $\tilde{\mf{v}}$ attains a constant value on the boundary $\pa \Omega_\infty$.
\end{itemize}
\end{lemma}

\begin{proof}
($i$) Let $\phi_i$ be the harmonic extension of $\varphi_i$ in $\Omega$, which is $C^{1,\gamma}(\overline{\Omega})$. By the comparison principle $0 \le u_{i,\beta} \le \phi_i$, for every $\beta$. Now, thanks to the Kirszbraun theorem, we can extend the functions $\varphi_i$ in the whole space $\R^N$ in a Lipschitz fashion, preserving their Lipschitz constant. The extended function will be still denoted by $\varphi_i$. We also extend $\phi_i$ and $u_{i,\beta}$ in $\R^N$, by letting them equal to $\varphi_i$ in $\Omega^c$. Let $\varphi_{i,\beta}$ and $\phi_{i,\beta}$ be given by scaling $\phi_i$ and $\varphi_i$ in the same way of $u_{i,\beta}$. We have that $v_{i,\beta}$, $\varphi_{i,\beta}$ and $\phi_{i,\beta}$ are defined everywhere, and $\varphi_{i,\beta} = \phi_{i,\beta} = v_{i,\beta}$ in $\Omega_\beta^c$. Plainly:
\begin{itemize}
\item[($i$)] $v_{i,\beta}$ is locally $\alpha$-H\"older continuous in $\R^N$, with $\alpha$-H\"older seminorm $[v_{i,\beta}]_{C^{0,\alpha}(K)}$ uniformly bounded with respect to $\beta$, for any compact set $K \subset \R^N$. 
\item[($ii$)] $\varphi_{i,\beta}$ and $\phi_{i,\beta}$ are locally Lipschitz continuous in $\R^N$, with Lipschitz seminorms $[\varphi_{i,\beta}]_{C^{0,1}(K)}$ and $[\phi_{i,\beta}]_{C^{0,1}(K)}$ uniformly bounded with respect to $\beta$, for any compact set $K \subset \R^N$.
\end{itemize}
Thus, since $\{v_{i,\beta}(0)\}$ is bounded, up to a subsequence $v_{i,\beta} \to v_i$ locally uniformly in $\R^N$. 
But $v_{i,\beta} = \varphi_{i,\beta} = \phi_{i,\beta}$ in $\Omega_\beta^c$, so that that $\varphi_{i,\beta} \to \varphi_{i,\infty}$ and $\phi_{i,\beta} \to \phi_{i,\infty}$ locally uniformly in $\Omega_\infty^c$. In turn, by uniform Lipschitz continuity, we infer that $\varphi_{i,\beta} \to \varphi_{i,\infty}$ and $\phi_{i,\beta} \to \phi_{i,\infty}$ locally uniformly in the whole of $\R^N$. The local uniform convergence entails $\varphi_{i,\infty} \cdot \varphi_{j,\infty} \equiv 0$ in $\overline{\Omega_\infty}$. Moreover $0 \le v_i \le \phi_{i,\infty}$ in $\Omega_\infty$.

Now we show that both $\varphi_{i,\infty}$ and $\phi_{i,\infty}$ are constant in $\R^N$, and, since they coincide in $\Omega_\infty^c$, they actually coincide everywhere. This is a consequence of the fact that $\varphi_{i,\infty}$ and $\phi_{i,\infty}$ are obtained as limits of scaling of a fixed Lipschitz continuous function, so that if $x \neq y$
\[
\frac{|\varphi_{i,\beta}(x)- \varphi_{i,\beta}(y)|}{|x-y|^\alpha} = \frac{|\varphi_{i}(x_\beta + r_\beta x)- \varphi_{i}(x_\beta + r_\beta y)|}{L_\beta r_\beta^\alpha |x-y|^\alpha} \le \frac{[\varphi_i]_{C^{0,1}(\R^N)}  r_\beta^{1-\alpha}} {L_\beta} |x-y|^{1-\alpha},
\]
and the right hand side tends to $0$ locally uniformly in $\R^N$. The very same argument proves that also $\phi_{i,\infty}$ is constant.

To sum up, so far we showed that the extended functions $v_{i,\beta}$, $\varphi_{i,\beta}$, $\phi_{i,\beta}$ converge locally uniformly in $\R^N$, coincide in $\Omega_\beta^c$, and $\varphi_{i,\infty} = \phi_{i,\infty}$ are constants in $\R^N$. Recalling the segregation condition $\varphi_{i,\infty} \cdot \varphi_{j,\infty} \equiv 0$ in $\Omega_\infty$, and hence also in $\R^N$, we deduce that at most one component $\phi_{i,\infty}$ can be different from $0$. But then, since $0 \le v_i \le \phi_{i,\infty}$ in $\Omega_\infty$, at most one component $v_i$ is different from $0$ in $\Omega_\infty$. And finally, since $v_i = \varphi_{i,\infty}$ in $\Omega_\infty^c$, we have that $v_i$ is constant on $\pa \Omega_\infty$.

The proof of ($ii$) is analogue.
\end{proof}

\begin{lemma}\label{lem: bdd 0}
Let $r_\beta \to 0^+$ be such that
\begin{itemize}
\item[($i$)] there exists $R'>0$ such that $|x_\beta-y_\beta| \le R' r_\beta$; 
\item[($ii$)] $M_\beta \not \to 0$.
\end{itemize}
Then $\{\mf{v}_\beta(0)\}$ is bounded in $\beta$.
\end{lemma}

\begin{proof}
The proof is analogue to the one of \cite[Lemma 6.1]{CTV05} (see also \cite[Lemma 3.4]{NTTV10} for more details), and hence we only sketch it. Suppose by contradiction that along a subsequence $v_{h,\beta}(0) \to +\infty$ for some index $h$, and let $R>R'$. By assumption ($ii$) and the global H\"older bound, we have that 
\[
I_\beta := M_\beta  \inf_{B_{2R} \cap \Omega_\beta} v_{h,\beta} \to +\infty.
\]
Now we can argue as in \cite[Lemma 6.1]{CTV05}, by using Lemma \ref{lem: cpq} instead of \cite[Lemma 4.4]{CTV05}, to deduce that for every $R>R'$
\[
\|v_{i,\beta}\|_{L^\infty(B_R \cap \Omega_\beta)} \to 0 \quad \forall i \neq h, \quad \text{and} \quad  \|L_i v_{i,\beta}\|_{L^\infty(B_R \cap \Omega_\beta)} \to 0 \qquad \forall i,
\] 
as $\beta \to +\infty$. Let then $\tilde{\mf{v}}_\beta(x) := \mf{v}_\beta(x)-\mf{v}_\beta(0)$. The above discussion shows that $\tilde{\mf{v}}_{\beta} \to \tilde{\mf{v}}$ locally uniformly in $\R^N$, where $\tilde{\mf{v}}$ is globally $\alpha$-H\"older continuous in $\Omega_\infty$, and $\tilde v_i \equiv 0$ for $i \neq h$ (in case $\Omega_\infty$ is a half-space, we can use Lemma \ref{lem: half-space}). The uniform convergence of the $A_i$-Laplacians implies that actually $\tilde{\mf{v}}_{\beta} \to \tilde{\mf{v}}$ in $C^1_{\loc}(\Omega_\infty)$. We claim that $\tilde v_1$ is not constant. To prove the claim, we recall that, by assumption ($i$), $(y_\beta-x_\beta)/r_\beta$ converges to a limit $z$ up to a subsequence. If $z=0$, by boundedness in $C^1_{\loc}$
\[
1 = \frac{|v_{1,\beta}(0)-v_{1,\beta}\left(\frac{y_\beta-x_\beta}{r_\beta}\right)|}{|\frac{y_\beta-x_\beta}{r_\beta}|^{\alpha}} =  \frac{|\tilde v_{1,\beta}(0)-\tilde v_{1,\beta}\left(\frac{y_\beta-x_\beta}{r_\beta}\right)|}{|\frac{y_\beta-x_\beta}{r_\beta}|^{\alpha}} \le C \left|\frac{y_\beta-x_\beta}{r_\beta}\right|^{1-\alpha} \to 0,
\]
a contradiction. Then $z \neq 0$, and $|\tilde v_1(0)- \tilde v_1(z)| = |z|^\alpha$,
so that $\tilde v_1$ is a non-constant $A_1$-harmonic function in $\Omega_\infty$, globally $\alpha$-H\"older continuous. If $\Omega_\infty$ is a half-space, by Lemma \ref{lem: half-space} we can also say that $\tilde v_1$ is constant on $\pa \Omega_\infty$. Therefore Lemma \ref{lem: liou ent} provides a contradiction both for $\Omega_\infty=\R^N$, and for $\Omega_\infty$ equal to a half-space.
\end{proof}

\begin{lemma}\label{lem: rel}
It results that 
\[
\limsup_{\beta \to +\infty} \beta L_\beta |x_\beta-y_\beta|^{2+\alpha} = +\infty
\]
\end{lemma}

\begin{proof}
By contradiction, let $\beta L_\beta |x_\beta-y_\beta|^{2+\alpha}$ be bounded. Then, by choosing 
\[
r_\beta = (\beta L_\beta)^{-\frac{1}{2+\alpha}},
\]
we have that $M_\beta=1$, and $|x_\beta-y_\beta| \le R' r_\beta$ for a constant $R'>0$. Hence $\{\mf{v}_\beta(0)\}$ is bounded, by Lemma \ref{lem: bdd 0}, and by uniform H\"older continuity $\mf{v}_\beta \to \mf{v}$ locally uniformly in $\Omega_\infty$, up to a subsequence. In addition, if $\Omega_\infty$ is a half-space, we know by Lemma \ref{lem: half-space} that each component of $\mf{v}$ but possibly one vanishes, and the remaining one is constant on $\pa \Omega_\infty$. Moreover, since $M_\beta=1$, we have that $L_i v_{i,\beta}$ converges locally uniformly, and hence $\mf{v}_\beta \to \mf{v}$ in $C^1_{\loc}(\Omega_\infty)$, with $\mf{v}$ globally $\alpha$-H\"older continuous in $\Omega_\infty$, and 
\[
-\div(A_i \nabla v_i ) = - v_i \sum_{j \neq i} b_{ij} v_j, \quad v_i \ge 0 \qquad \text{in $\Omega_\infty$}.
\]
In fact either $v_i>0$, or $v_i \equiv 0$ in $\Omega_\infty$, by the strong maximum principle. Finally, as in the last part of the proof of Lemma \ref{lem: bdd 0}, we deduce also that $v_1$ is non-constant in $\Omega_\infty$. 

Let $\Omega_\infty = \R^N$. Then, since $2\alpha < \bar \nu$, by Theorem \ref{thm: liou syst}\footnote{In the present case, each $I_h$ is a singleton, and the assumptions on the coupling terms $g_{ij}$ are satisfied since $b_{ij}>0$. Notice also that the $\alpha$-H\"older continuity implies that $v_1,\dots,v_k$ grow at most like $|x|^\alpha$ at infinity} we have that $\mf{v}$ is constant, a contradiction. 

If instead $\Omega_\infty$ is a half-space, by Lemma \ref{lem: half-space} we know that at most one component $v_i$ does not vanish identically. Since $v_1$ is non-constant, we infer that $v_i \equiv 0$ for every $i \neq 1$, and $v_1$ is a non-constant $A_1$-harmonic function in a half-space, globally $\alpha$-H\"older continuous, which attains a constant boundary datum on $\pa \Omega_\infty$. This gives a contradiction with Lemma \ref{lem: liou ent}.
\end{proof}

At this point we fix the choice of $r_\beta$ and complete the contradiction argument.

\begin{proof}[Conclusion of the proof of Theorem \ref{thm: lv}]
Let 
\[
r_\beta= |x_\beta-y_\beta|.
\]
By Lemma \ref{lem: rel}, we have that $M_\beta \to +\infty$. Thus the assumptions of Lemma \ref{lem: bdd 0} are satisfied, and we deduce that $\mf{v}_\beta \to \mf{v}$ locally uniformly in $\R^N$, with $\mf{v}$ globally $\alpha$-H\"older continuous (if $\Omega_\infty$ is a half-space, we consider the extension of $\mf{v}_\beta$ defined in Lemma \ref{lem: half-space}; in this case, we have that $\mf{v}$ is constant on $\pa \Omega_\infty$ and has at most one non-trivial component). Furthermore, there exists $z \in \pa B_1 \cap \overline{\Omega_\infty}$ such that $|v_1(z) - v_1(0)|=1$, and hence $v_1$ is non-constant. 

Now, let $r>0$ and $x_0$ be such that $B_{2r}(x_0) \subset \subset \Omega_\infty$, and let $\eta$ be a smooth cut-off function such that $0\le \eta \le 1$, $\eta \equiv 1$ in $B_r(x_0)$, and $\eta \equiv 0$ in $B_r(x_0)^c$. By testing the equation for $v_{i,\beta}$ with $\eta$, we deduce that 
\begin{equation}\label{1031}
M_\beta \int_{B_r(x_0)} v_{i,\beta} \sum_{j \neq i} b_{ij} v_{j,\beta} \le \int_{B_{2r}(x_0)} \div(A_i \nabla \eta) v_{i,\beta} \le C,
\end{equation}
since $\{v_{i,\beta}\}$ is locally bounded in $L^\infty$. But $M_\beta \to +\infty$, so that $v_i \cdot v_j \equiv 0$ in $\Omega_\infty$. 

Moreover, by testing the equation for $v_{i,\beta}$ with $v_{i,\beta} \eta^2$, we obtain
\begin{equation}\label{1032}
\begin{split}
\int_{B_{r}(x_0)} \langle A_i \nabla v_{i,\beta}, \nabla v_{i,\beta} \rangle & \le 4\int_{B_{2r}(x_0)} \langle A_i \nabla \eta, \nabla \eta \rangle v_{i,\beta}^2 +2 M_\beta \int_{B_{2r}(x_0)} v_{i,\beta}^2 \sum_{j \neq i} b_{ij} v_{j,\beta}\\
& \le  \|v_{i,\beta}\|_{L^\infty(B_{2r(x_0)})} \left(C + M_\beta \int_{B_{2r}(x_0)} v_{i,\beta} \sum_{j \neq i} b_{ij} v_{j,\beta}\right) \le C,
\end{split}
\end{equation}
where we used \eqref{1031}.  That is, $\{\mf{v}_\beta\}$ is locally bounded in $H^1_{\loc}(\Omega_\infty)$ and hence, up to a further subsequence, $\mf{v}_\beta \weak \mf{v}$ weakly in $H^1_{\loc}(\Omega_\infty)$. Since $-\div(A_i \nabla v_{i,\beta}) \le 0$ in $\Omega_\beta$, if we take the weak limit we infer that
\[
v_i \cdot v_j \equiv 0 \quad \text{if $i \neq j$}, \quad \text{and} \quad -\div(A_i \nabla v_i) \le 0 \quad \text{in $\Omega_\infty$}.
\]
If $\Omega_\infty$ is a half-space, Lemma \ref{lem: half-space} also implies that each component of $\mf{v}$ but $v_1$ must vanish identically. If instead $\Omega=\R^N$, the same conclusion follows from Theorem \ref{thm: liou seg}, since $2\alpha < \bar \nu$. In any case, for every $\beta$
\[
-\div(A_1 \nabla v_{1,\beta})+ \sum_{j =2}^k \frac{b_{1j}}{b_{j1}} \div(A_j \nabla v_{j,\beta}) = M_\beta \sum_{j=2}^k \, \sum_{h \neq 1,j} \frac{b_{1j} b_{jh} }{b_{j1} } v_{j,\beta} v_{h,\beta} \ge 0 \quad \text{in $\Omega_\beta$} 
\]
in weak sense. By passing to the weak limit, and recalling that $v_j \equiv 0$ in $\R^N$ for $j \neq 1$, we deduce that 
\begin{equation}\label{1033}
-\div(A_1 \nabla v_1) \ge 0 \qquad \text{in $\Omega_\infty$}
\end{equation}
in weak sense. But then $\div(A_1 \nabla v_1) = 0$ in $\Omega_\infty$, and Lemma \ref{lem: liou ent} gives a contradiction with the fact that $v_1$ is $\alpha$-H\"older continuous and non-constant in $\R^N$. This contradiction finally shows that $\{\mf{u}_\beta\}$ is bounded in $C^{0,\alpha}(\overline{\Omega})$, as desired. Now we proceed with the second part of the theorem. 

Clearly, we have that up to a subsequence $\mf{u}_\beta \to \mf{u}$ in $C^{0,\alpha}(\overline{\Omega})$, for every $\alpha \in (0,\bar \nu/2)$. As in \eqref{1031} and \eqref{1032} it is possible to check that $\mf{u}_\beta \weak \mf{u}$ weakly in $H^1_{\loc}(\Omega)$, and that the limit function is segregated: $u_i \cdot u_j \equiv 0$ for $i \neq j$. Moreover, by testing the equation for $u_{i,\beta}$ with $(u_{i,\beta}-u) \eta$, where $\eta \in C^\infty_c(\Omega)$ is an arbitrary cut-off function, we deduce that
\[
\begin{split}
\int_{\Omega} \eta \langle A_i \nabla u_{i,\beta}, \nabla (u_{i,\beta}-u) \rangle &= - \int_{\Omega} (u_{i,\beta}-u_i) \langle A_i\nabla u_{i,\beta}, \nabla \eta \rangle - \beta \int_{\Omega} \eta u_{i,\beta} \sum_{j \neq i} b_{ij} u_{j,\beta} (u_{i,\beta}-u_{j,\beta}) \\
& \le C\|u_{i,\beta}- u_i\|_{L^\infty(\Omega)} \left( \| \nabla u_{i,\beta}\|_{L^2(\supp \eta)} + \beta \int_{\Omega} u_{i,\beta} \sum_{j \neq i} b_{ij} u_{j,\beta} \right) \to 0
\end{split}
\]
as $\beta \to \infty$, by uniform convergence and local boundedness in $H^1$. Therefore, by weak convergence
\[
\begin{split}
0 &= \lim_{\beta \to +\infty} \int_{\Omega} \eta \langle A_i \nabla u_{i,\beta}, \nabla (u_{i,\beta}-u) \rangle \\
& = \lim_{\beta \to +\infty} \int_{\Omega} \eta \big(  \langle A_i \nabla u_{i,\beta}, \nabla u_{i,\beta} \rangle -  \langle A_i \nabla u_i, \nabla u_i \rangle \big),
\end{split}
\]
which gives the strong convergence $\mf{u}_\beta \to \mf{u}$ in $H^1_{\loc}(\Omega)$. Finally, to show that $u_i$ is $A_i$-harmonic in $\{u_i>0\}$, we proceed as for \eqref{1033}, proving that $\div(A \nabla u_i) \le 0$ in $\{u_i>0\}$. But by weak convergence $\div(A_i \nabla u_i) \ge 0$ in $\Omega$, and the thesis follows.
\end{proof}

\begin{proof}[Proof of Theorem \ref{thm: limit lv}]
Recall that $k=2$, and we reduced to the case when $A_1=A$ is diagonal, with lowest eigenvalue equal to $1$, $A_2=\textrm{Id}$, and $a_{12}=a_{21}$. We use the notation $(u_{1,\beta}, u_{2,\beta}) = (u_\beta, v_\beta)$. Let us consider $w_\beta = u_\beta-v_\beta$. We know that, up to a subsequence, $w_\beta \to w = u-v$ in $C^{0,\alpha}(\overline{\Omega}) \cap H^1_{\loc}(\Omega)$ for every $0<\alpha< \frac{\bar \nu}2= \frac{\nu_{A,N}}2$. On the other hand, since $\div(A_i \nabla u_\beta) = \Delta v_\beta$, we have that
\[
\int_{\Omega} \langle A \nabla u_\beta, \nabla \varphi \rangle - \langle \nabla v_\beta, \nabla \varphi \rangle =0
\]
for every $\varphi \in C^1_c(\Omega)$. By taking the limit, by weak convergence in $H^1$ we deduce that $w = u-v$ is a weak solution of the quasi-linear equation \eqref{qu eq}. The rest of the theorem follows directly from the main result in \cite{AM12}.
\end{proof}

\section{Spatial segregation of competitive systems: variational interaction}\label{sec: be}

In this section we prove Theorem \ref{thm: be}. 
Notice that each $\mf{u}_\beta$ is a vector of positive functions in $\Omega$, of class $C^{1,\gamma}(\overline{\Omega})$. Moreover, we can define $\bar \nu= \bar \nu(N, A_1,\dots,A_k) \in (0,2)$ as in Theorem \ref{thm: liou seg}.


Now, the first part of the proof of Theorem \ref{thm: be} rests upon the same contradiction argument used for Theorem \ref{thm: lv}, with the obvious modifications related to the different structure for the system, and to the different boundary conditions. We only give a sketchy summary referring, for the details, to the previous section and to \cite{NTTV10, STTZ16} (of course, we use Theorems \ref{thm: liou seg}, \ref{thm: liou syst}, and Lemma \ref{lem: liou ent} instead of the corresponding ``symmetric results" when it is necessary). 

Let $\alpha \in (0,\bar \nu/2)$, and suppose by contradiction that $\{\mf{u}_\beta\}$ is not bounded in $C^{0,\alpha}(\overline{\Omega})$, namely there exists a sequence $\beta \to +\infty$ such that 
\[
L_\beta:= \sup_i \sup_{\substack{x \neq y, \\ x, y \in \overline{\Omega}}} \frac{|u_{i,\beta}(x)- u_{i,\beta}(y)|}{|x-y|^\alpha}  \to +\infty.
\]
We can assume that $M_\beta$ is achieved by $u_{1,\beta}$ at the pair $(x_\beta, y_\beta)$, with $|x_\beta -y_\beta| \to 0$. Then we introduce the following blow-up of $\mf{u}_\beta$ with center in $x_\beta$, and $r_\beta \to 0^+$ to be chosen later:
\[
\mf{v}_\beta(x):= \frac{1}{L_\beta r_\beta^\alpha} \mf{u}_\beta(x_\beta + r_\beta x), \qquad x \in \Omega_\beta := \frac{\Omega-x_\beta}{r_\beta}.
\]
The scaled domains $\Omega_\beta$ can either exhaust $\R^N$, or tend to a half-space. In both cases, we denote the limit domain by $\Omega_\infty$. The function $\mf{v}_\beta$ is a positive solution to
\[
\begin{cases}
-L_i v_{i,\beta}= g_{i,\beta}(x,v_{i,\beta})- M_\beta v_{i,\beta} \sum_{j \neq i} b_{ij} v_{j,\beta}^2  &\text{ in $\Omega_\beta$}\\
v_{i,\beta} = 0 &\text{ on $\pa \Omega_\beta$},
\end{cases} 
\]
where $M_\beta = L_\beta^2 r_\beta^{2+2\alpha} \beta$, and 
\[
g_{i,\beta}(x,v_{i,\beta}(x)) = \frac{r_\beta^{2-\alpha}}{L_\beta} f_{i,\beta}\left(x_\beta + r_\beta x, L_\beta r_\beta^\alpha v_{i,\beta}(x) \right) = \frac{r_\beta^{2-\alpha}}{L_\beta} f_{i,\beta}\left(x_\beta + r_\beta x,u_{i,\beta}(x_\beta + r_\beta x) \right).
\]
Notice that $\|g_{i,\beta}(\cdot,v_{i,\beta}(\cdot))\|_{L^\infty(\Omega_\beta)} \to 0$ as $\beta \to +\infty$, thanks to the assumptions on $f_{i,\beta}$ and the upper bound on $\|u_{i,\beta}\|_{L^\infty(\Omega)}$. Moreover, for every $\beta$
\[
\max_i \max_{\substack{x \neq y, \\ x, y \in \overline{\Omega}}} \frac{|v_{i,\beta}(x)- v_{i,\beta}(y)|}{|x-y|^\alpha} =  \frac{|v_{1,\beta}(0)- v_{1,\beta}\left(\frac{y_\beta-x_\beta}{r_\beta} \right)|}{\left|\frac{x_\beta-y_\beta}{r_\beta}\right|^\alpha} = 1.
\]

\begin{lemma}\label{lem: half-space'}
Suppose that $\Omega_\beta$ tends to a half-space $\Omega_\infty$. Then it is possible to extend $\mf{v}_\beta$ outside $\Omega_\beta$ in a Lipschitz fashion, in such a way that:
\begin{itemize}
\item[($i$)] If $\{\mf{v}_\beta(0)\}$ is bounded, then 
$\mf{v}_\beta \to \mf{v}$ in $C^{0,\alpha'}_{\loc}(\R^N)$ for every $0<\alpha'<\alpha$, up to a subsequence; moreover, the limit function $\mf{v}$ attains the constant value $0$ on $\pa \Omega_\infty$.
\item[($ii$)] If $\{\mf{v}_\beta(0)\}$ is unbounded, then $\Omega_\infty=\R^N$. 
\end{itemize}
\end{lemma}
\begin{proof}
($i$) This is very similar to point ($i$) of Lemma \ref{lem: half-space}, once that we extend $\mf{u}_\beta$ as equal to $0$ outside $\Omega$. \\
($ii$) Let $R>0$ be arbitrarily chosen. If $v_{i,\beta}(0) \to +\infty$ along a subsequence, then by uniform H\"older estimates 
\[
\inf_{B_R} v_{i,\beta} \ge v_{i,\beta}(0) - CR^\alpha \to +\infty.
\]
But $v_{i,\beta} \equiv 0$ in $\Omega_\beta^c$, and hence $B_R(0) \subset \Omega_\beta$ eventually. 
\end{proof}

With the help of this lemma, and by following \cite[Section 3]{NTTV10} (see also Section \ref{sec: lv}), it is not difficult to prove that:

\begin{lemma}\label{lem: bdd 0'}
Let $r_\beta \to 0^+$ be such that
\begin{itemize}
\item[($i$)] there exists $R'>0$ such that $|x_\beta-y_\beta| \le R' r_\beta$; 
\item[($ii$)] $M_\beta \not \to 0$.
\end{itemize}
Then $\{\mf{v}_\beta(0)\}$ is bounded in $\beta$.
\end{lemma}

\begin{lemma}\label{lem: rel'}
It results that
\[
\limsup_{\beta \to +\infty} \beta L_\beta^2 |x_\beta-y_\beta|^{2+2\alpha} = +\infty.
\]
\end{lemma}

At this point we fix the choice of $r_\beta$ as in Section \ref{sec: lv}, and analyze the asymptotic behavior of $\mf{v}_\beta$. 

\begin{lemma}\label{lem: asy}
Let 
\[
r_\beta= |x_\beta-y_\beta|.
\]
There exist $\mf{v}$, globally $\alpha$-H\"older continuous with $[\mf{v}]_{C^{0,\alpha}(\R^N)} = 1$ such that, as $\beta \to +\infty$, the following holds up to a subsequence:
\begin{itemize}
\item[($i$)] $\Omega_\infty=\R^N$, and $\mf{v}_\beta \to \mf{v}$ in $C^{0,\alpha'}_{\loc}(\R^N)$;
\item[($ii$)] $\int_{B_r(x_0)} M_\beta v_{i,\beta}^2 v_{j,\beta}^2 \to 0$, for any $r>0$ and $x_0 \in \R^N$;
\item[($iii$)] $\mf{v}_\beta \to \mf{v}$ in $H^1_{\loc}(\R^N)$.
\end{itemize}
\end{lemma}

For the proof, we refer to \cite[Lemmas 3.6 and 3.7]{NTTV10} (see also the conclusion of the proof of Theorem \ref{thm: lv}). The properties of the limit profile are collected in next statement.

\begin{lemma}\label{lem: prop limit}
Let $\mf{v}$ be the limit function defined in Lemma \ref{lem: asy}. Then:
\begin{itemize}
\item[($i$)] $v_i \equiv 0$ in $\R^N$ for every $i \neq 1$;
\item[($ii$)] $v_1$ is non-constant, and $\div (A_1 \nabla v_1) = 0$ in $\{v_1>0\}$;
\item[($iii$)] $\{v_1=0\} \neq \emptyset$ and $\{v_1>0\}$ is connected.
\end{itemize}
\end{lemma}

\begin{proof}
By Lemma \ref{lem: asy} we have that $v_i \cdot v_j \equiv 0$ in $\R^N$. Moreover,  by $H^1_{\loc}$ convergence and recalling that $\|g_{i,\beta}(\cdot, v_{i,\beta}(\cdot))\|_{L^\infty(\Omega_\beta)} \to 0$, we deduce that $v_i$ is $A_i$-subharmonic, for every $i$. Since $\mf{v}$ is $\alpha$-H\"older with $\alpha < \bar \nu/2$, Theorem \ref{thm: liou seg} implies that only one component of $\mf{v}$ does not vanish identically. But by uniform convergence $\max_{x \in \pa B_1(0)} |v_1(x)- v_1(0)| = 1$, so that $v_i \equiv 0$ in $\R^N$ for every $i \neq 1$, and $v_1$ is non-constant. The fact that $v_1$ is harmonic in the open set $\{v_1>0\}$ can be checked as in \cite[Lemma 3.7]{NTTV10}, by using Lemma \ref{lem: cpq} instead of \cite[Lemma 3.1]{NTTV10}. This completes the proof of ($i$) and ($ii$). 

Suppose now by contradiction that $\{v_1=0\}$ is empty; then $v_1$ would be a positive globally $\alpha$-H\"older $A_1$-harmonic non-constant function, in contradiction Lemma \ref{lem: liou ent}. 

Finally, suppose by contradiction that $\{v_1>0\}$ is disconnected, and let $\omega_1$ and $\omega_2$ two of its connected components. Then the functions $w_1= v_1 \chi_{\omega_1}$ and $w_2 = v_1 \chi_{\omega_2}$ are non-trivial and satisfy the assumptions of Theorem \ref{thm: liou seg}, a contradiction again.   
\end{proof}

From now on we shall mainly focus on the component $v_1$, the only one which survived in the limit process. Therefore, in order to simplify some expressions below, we perform a rotation and a scaling in order to have $A_1= \textrm{Id}$, and hence $v_1$ is harmonic in its positivity set.

\begin{remark}
In the conclusion of the proof of Theorem \ref{thm: lv}, we considered the difference among the differential equations of the component $v_{1,\beta}$ and the others, by taking the weak limit. This gave us the inequality $\div(A_1 \nabla v_1) \ge 0$ in $\R^N$, leading to the $A_1$-harmonicity of $v_1$, which finally provided a contradiction. When we deal with system \ref{be completa}, this strategy fails, due to the lack of symmetry in the exponents of the competition terms. By following \cite{NTTV10}, one may be tempted to consider an Almgren frequency function $N_\beta(\mf{v}_\beta,x_0,r)$ associated with $\mf{v}_\beta$, compute its derivative, and then pass to the limit in $\beta$ in order to derive a monotonicity formula for the frequency function of the limit problem (see \cite[Section 3.2]{NTTV10}). However, in the present setting this strategy fails, due to the lack of symmetry in the diffusion operators. This lack of symmetry creates several complications in the derivation of a good expression for the derivative of $N_\beta(\mf{v}_\beta,x_0,r)$, complications which we could not overcome. We shall then argue in a different way. First, by the variational structure of the problem (this requires $b_{ij}=b_{ji}$), we derive a domain variation formula for $\mf{v}_\beta$. Then we pass to the limit in $\beta$. The properties collected in Lemmas \ref{lem: asy} and \ref{lem: prop limit} at this level allow us to obtain the validity of a domain variation formula for the only non-trivial component $v_1$, in the whole of $\R^N$ (and not only in the interior of its support). In this way, even if we cannot establish a monotonicity formula for the Almgren frequency function associated with $\mf{v}_\beta$, we can still recover a monotonicity formula for the component $v_1$ of the limit profile. This is sufficient to our purposes.
\end{remark}

\begin{lemma}\label{lem: var dom}
Let $Y \in C^\infty_c(\R^N, \R^N)$. Then
\begin{equation}\label{var dom beta}
\begin{split}
2\int_{\Omega_\beta}  \sum_i \langle dY A_i \nabla v_{i,\beta}, \nabla v_{i,\beta} \rangle &-  \sum_i g_{i,\beta}(x,v_{i,\beta}) \langle \nabla v_{i,\beta},Y \rangle \\
& - \int_{\Omega_\beta} \div Y \left( \sum_i \langle A_i \nabla v_{i,\beta}, \nabla v_{i,\beta} \rangle + \beta \sum_{i<j} b_{ij} v_{i,\beta}^2 v_{j,\beta}^2\right)   = 0
\end{split}
\end{equation}
for every $\beta$ sufficiently large, and
\begin{equation}\label{var dom lim}
\int_{\R^N} 2\langle dY \nabla v_1, \nabla v_1 \rangle - \div Y |\nabla v_1|^2 = 0.
\end{equation}
\end{lemma}

\begin{proof}
By multiplying the equation for $v_{i,\beta}$ with $\langle \nabla v_{i,\beta}, Y \rangle$ and integrating, we deduce that 
\[
\int_{\Omega_\beta} \langle A_i \nabla v_{i,\beta}, \nabla \big( \langle \nabla v_{i,\beta}, Y \rangle \big) \rangle + \beta v_{i,\beta} \sum_{j \neq i} b_{ij} v_{j,\beta}^2 \langle \nabla v_{i,\beta}, Y \rangle - g_{i,\beta}(x,v_{i,\beta}) \langle \nabla v_{i,\beta}, Y \rangle=0
\]
With lengthy but elementary computations, it is not difficult to check that
\[
\langle A_i \nabla v_{i,\beta}, \nabla \big( \langle \nabla v_{i,\beta}, Y \rangle \big) \rangle = \frac12 \langle Y, \nabla \big(\langle A_i \nabla v_{i,\beta}, \nabla v_{i,\beta} \rangle\big) \rangle + \langle dY A_i \nabla v_{i,\beta}, \nabla v_{i,\beta} \rangle,
\]
whence, by integrating by parts, we deduce that 
\[
\begin{split}
\int_{\Omega_\beta} \langle dY A_i \nabla v_{i,\beta}, \nabla v_{i,\beta} \rangle &-  g_{i,\beta}(x,v_{i,\beta}) \langle \nabla v_{i,\beta}, Y \rangle \\
&- \int_{\Omega_\beta} \frac12 \div Y \langle A \nabla v_{i,\beta}, \nabla v_{i,\beta} \rangle -   \beta v_{i,\beta} \sum_{j \neq i} b_{ij} v_{j,\beta}^2 \langle \nabla v_{i,\beta}, Y \rangle = 0.
\end{split}
\]
We sum over $i$ from $1$ to $k$ and integrate by parts once again, to obtain \eqref{var dom beta}.

Moreover, by taking the limit as $\beta \to +\infty$, and recalling the properties listed in Lemmas \ref{lem: asy} and \ref{lem: prop limit}, and the fact that $\|g_{i,\beta}(\cdot,v_{i,\beta}(\cdot))\|_{L^\infty(\Omega_\beta)} \to 0$, we also obtain \eqref{var dom lim}.
\end{proof}

\begin{proof}[Conclusion of the proof of Theorem \ref{thm: be}]
From the second formula in Lemma \ref{lem: var dom}, we infer that
\[
\int_{S_r(x_0)} |\nabla v_1|^2 = \frac{N-2}r \int_{B_r} |\nabla v_1|^2 + 2 \int_{S_r(x_0)} (\pa_\nu v_1)^2
\]
for every $x_0 \in \R^N$ and almost every $r>0$ (we refer to the second part of the proof of Lemma 2.11 in \cite{ST19} for the details). Furthermore, we have that
\[
\frac{d}{dr} \left( \int_{S_r(x_0)} v_1^2\right) = \frac{N-1}{r} \int_{S_r(x_0)} v_1^2 + 2 \int_{S_r(x_0)} v_1 \pa_\nu v_1
\]
(see \cite[Lemma 2.8]{ST19} for the details). Therefore, by introducing the Almgren frequency function
\[
N(x_0,r) =  \frac{r \int_{B_r(x_0)} |\nabla v_1|^2}{\int_{S_r(x_0)} v_1^2},
\]
it is standard to prove that $N(x_0,\cdot)$ is monotone non-decreasing, and it is constant equal to $c$ if and only $v_1$ is $c$ homogeneous. At this point we can proceed exactly as in \cite[End of the proof of Theorem 1.3]{NTTV10} or \cite[Conclusion of the proof of Proposition 2.1, page 278]{ST19} to deduce that $\{v_1=0\}$ is a linear subspace of dimension at most $N-2$, and in particular has local capacity equal to $0$. But then, since $v \in H^1_{\loc}(\R^N)$, we infer that $v_1$ is harmonic everywhere, is non-constant, and is globally $\alpha$-H\"older continuous for some $\alpha \in (0,1)$, in contradiction with the Liouville theorem. This completes the proof of the boundedness of $\{\mf{u}_\beta\}$ in $C^{0,\alpha}(\overline{\Omega})$. The rest of the thesis of Theorem \ref{thm: be} follows as in \cite{NTTV10} (for the domain variation formula \eqref{dom var lim 1}, one can argue as in Lemma \ref{lem: var dom} with the functions $\mf{u}_\beta$, and then take the limit). 
\end{proof}


\end{document}